\documentclass[12pt,reqno]{amsart}
\usepackage{amsfonts,amsmath}
\usepackage{amssymb}
\usepackage{textcomp}
\usepackage{enumitem}
\usepackage{graphicx,color}
\numberwithin{equation}{section}
\allowdisplaybreaks

\newcommand{\nc}{\newcommand}
\nc{\parent}[1]{$[\![#1]\!]$}
\newtheorem{theorem}{Theorem}[section]
\newtheorem{lemma}{Lemma}[section]
\newtheorem{example}{Example}[section]
\newtheorem{corollary}{Corollary}[section]
\newtheorem{proposition}{Proposition}[section]
\newtheorem{remark}{Remark}[section]
\newtheorem{definition}{Definition}[section]
\newtheorem{assumption}{Assumption}[section]

\newcommand{\Implies}[2]{$\text{\ref{#1}}\implies\text{\ref{#2}}$}
\newenvironment{pf-main}{{\sc Proof of Theorem \ref{mainresult}.}\hspace{3mm}}{\qed}

\nc{\cadlag}{c\`{a}dl\`{a}g } \nc{\Ito}{It\^o} \nc{\ba}{\begin{array}}\nc{\elr}{(\ell,r)}
\nc{\ea}{\end{array}} \nc{\be}{\begin{equation}}
\nc{\ee}{\end{equation}} \nc{\bea}{\begin{eqnarray}}
\nc{\eea}{\end{eqnarray}} \nc{\bean}{\begin{eqnarray*}}
\nc{\eean}{\end{eqnarray*}} \nc{\bu}{\bullet} \nc{\nn}{\nonumber}
\nc{\cA}{{\mathcal A}} \nc{\cB}{{\mathcal B}} \nc{\cE}{\mathcal{E}}\nc{\cC}{{\mathcal
C}} \nc{\cD}{{\mathcal D}} \nc{\bbD}{\mathbb{D}}
\nc{\cG}{{\mathcal G}} \nc{\cF}{{\mathcal F}} \nc{\cS}{{\mathcal
S}} \nc{\cU}{{\mathcal U}} \nc{\cH}{{\mathcal H}}
\nc{\cK}{{\mathcal K}}\nc{\cL}{{\mathcal L}}  \nc{\cM}{{\mathcal
M}} \nc{\cO}{{\mathcal O}} \nc{\cP}{{\mathcal P}} \nc{\bfE}{\mathbf{E}}
\nc{\bbE}{\mathbb{E}} \nc{\tbA}{\tilde{\bbA}}\nc{\bbA}{\mathbb{A}}\nc{\bbF}{\mathbb{F}}
\nc{\bbEQ}{\mathbb{E}_{\mathbb{Q}}} \nc{\eps}{\varepsilon}
\nc{\bbEP}{\mathbb{E}_{\mathbb{P}}}\nc{\bbL}{\mathbb{L}}
\nc{\what}{\widehat} \nc{\bbP}{\mathbb{P}} \nc{\bbQ}{\mathbb{Q}}
\nc{\del}{\partial} \nc{\Om}{\Omega} \nc{\om}{\omega}
\nc{\bbR}{\mathbb{R}} \nc{\bbN}{\mathbb{N}} \nc{\fps}{$(\Om, \cF,
(\cF_t)_{t\geq 0}, \bbP)$} \nc{\bbC}{\mathbb{C}}
\nc{\bfr}{\begin{flushright}} \nc{\efr}{\end{flushright}}
\nc{\dXt}{\Delta X_{t}} \nc{\dXs}{\Delta X_{s}}
\nc{\bs}{\blacksquare} \nc{\dX}{\Delta X} \nc{\dY}{\Delta Y}
\nc{\dnkx}{\left(X(T^{n}_{k})-X(T^{n}_{k-1})\right)}
\nc{\esssup}{\mathrm{ess}\mbox{ }\mathrm{sup}}
\nc{\essinf}{\mathrm{ess}\mbox{ } \mathrm{inf}}
\nc{\dhats}{\widehat{\delta_s}} \nc{\half} {\frac{1}{2}}

\nc{\ol}{\overline}
\def\rar{\rightarrow}

\nc{\chf}{\mbox{$\mathbf1$}}
\begin{document}

\title[Positive subharmonic functions and optimal stopping]{Integral representation of subharmonic functions and optimal stopping with random discounting}
\author{Umut \c{C}et\.in}
\address{Department of Statistics, London School of Economics and Political Science, 10 Houghton st, London, WC2A 2AE, UK}
\email{u.cetin@lse.ac.uk}
\date{\today}
\begin{abstract}
An integral representation result for strictly positive subharmonic functions of a one-dimensional regular diffusion is established. More precisely, any such function can be written as a linear combination of an increasing and a decreasing subharmonic function that solve an integral equation 
\[
g(x)=a  + \int v(x,y)\mu_A(dy) + \kappa s(x),
\] 
where $a>0$, $\kappa \in \bbR$, $s$ is a scale function of the diffusion, $\mu_A$ is a Radon measure, and $v$ is a kernel that is explicitly determined by the scale function. This integral equation in turn allows one construct a pair $(g,A)$ such that $g$ is a subharmonic function, $A$ is a continuous additive functional with Revuz measure $\mu_A$ and $g(X)\exp(-A)$ is a local martingale. The changes of measures associated with such pairs are studied and shown to modify the long term behaviour of the original diffusion process to exhibit transience. Theory is illustrated via examples that in particular contain a sequence of measure transformations that render the diffusion irregular in the limit by breaking the state space into distinct regions with soft and hard borders. Finally, the theory is applied to find an ``explicit'' solution to an optimal stopping problem with random discounting. 
\end{abstract}
\maketitle

\section{Introduction}
One of the fundamental results in the potential theory of Markov processes is the Riesz representation of an excessive (non-negative superharmonic) function as the sum of a harmonic function and the potential of a measure (see, e.g., Section VI.2 in \cite{BG}, \cite{Duncan71} and \cite{ChungRao80} for proofs under various assumptions). In the particular setting of a regular transient one-dimensional diffusion this amounts to a finite excessive function $f$ having the following representation:
\[
f(x)=\int u(x,y)\mu(dy)+ h(x),
\]
where $h$ is a harmonic function, $u$ is the potential density, and $\mu$ is a Borel measure.

On the other hand, analogous representation results for non-negative subharmonic functions of a given Markov processes do not seem to exist in a general form. If the Markov process is transient and the subharmonic function $g$ is bounded by $K$, one can obtain a representation using the available theory for the excessive function $K-g$. However, this approach will fail when $g$ is unbounded or the Markov process is recurrent, which implies all excessive functions are constant. Non-negative subharmonic functions are also known as `defective' functions (see p.31 of Dellacherie and Meyer \cite{DM-C}) and play an important role in Rost's solution to the Skorokhod embedding problem \cite{rost_stopping71}. Despite their abundance relative to excessive functions and their use in the potential theory, as Dellacherie and Meyer point out in \cite{DM-C}, ``It is quite depressing to admit that one knows almost nothing about defective functions.'' 

The main purpose of this paper is to fill a gap in this direction by establishing an integral representation for strictly positive subharmonic functions of a regular one-dimensional diffusion $X$ on a given interval $\elr$. It is shown in Theorem \ref{t:representation} that any such subharmonic function can be written as a linear combination of monotone subharmonic functions that are solutions of 
\be \label{e:intro:IE}
g(x)=g(c)+ \kappa (s(x)-s(c))+ \int_{\ell}^r v_c(x,y)g(y)\mu_A(dy),
\ee
where $\kappa\in \bbR$, $c \in \elr$, $s$ is a scale function, $\mu_A$ is a Radon measure, and  
\[
v_c(x,y)= s(x\vee y)-s(c\vee y)\; \mbox{or } s(c\wedge y)-s(x\wedge y).
\]
Conversely, solutions of (\ref{e:intro:IE}) can be used to construct strictly positive subharmonic functions. A family of integral equations for which solutions exist and can be used to generate {\em all} strictly positive subharmonic functions are studied in Section \ref{s:existence}.

To every strictly positive subharmonic function one can associate a {\em continuous additive functional} (CAF) $A$ such that $g(X)\exp(-A)$ is a local martingale. Section \ref{s:transform} studies changes of measures (or {\em path transformations}) for diffusions via such local martingales. It is in particular shown that after these path transformations  the diffusion process ends up transient, thereby providing  a complete counterpart to {\em recurrent transformations} introduced in \cite{rectr} via $h(X)\exp(B)$, where $h$ is excessive and $B$ is a CAF. 

The theory developed in Sections \ref{s:IE}-\ref{s:transform} is illustrated via Examples in Section \ref{s:examples}. In particular a  connection with the fundamental solutions of ordinary differential equations is made when the measure $\mu_A$ in (\ref{e:intro:IE}) is absolutely continuous with respect to the Lebesgue measure. Furthermore, in a series of remarkable examples a sequence of measure transformations via $g(X)\exp(-A)$, where $A$ is a mixture of local time processes, are shown to break the state space into several regions in the limit with particular borders. Roughly speaking, one can identify two different types of border behaviour in the limit: i) a {\em soft border} that allows a one-way passage between two neighbouring regions and ii) a {\em hard border} not allowing any interaction between the neighbours. Although the limiting process is no longer regular in the sense that it is not possible to reach some sets starting from some other sets, each path transformation results in a regular diffusion. However, they display an almost reflective behaviour at certain points in the interior of the state space that will later correspond to soft and hard borders in the limit.  Albeit tempting, this intriguing asymptotic behaviour  deserves a treatment in its own right and is therefore left to future work for a detailed analysis.

The path transformations introduced in Section \ref{s:transform} is used to solve a version of the optimal stopping problem with random discounting studied earlier by \cite{BL} and \cite{dayanikRD}. The problem of interest is to solve
\[
\sup_{\tau}E^x[e^{-A_{\tau}}f(X_{\tau})],
\]
where $f$ is a reward function and $A$ is a CAF. The method presented here is very similar at heart to the approach first proposed by Beibel and Lerche \cite{BL} and later developed in further generality by Dayan\i k in \cite{dayanikRD} for one-dimensional diffusions. The main idea in all these works - including the one presented here - is to find a subharmonic function $g$ so that $g(X)\exp(-A)$ is a local martingale. This allows for the reduction of the above optimal stopping problem to one without discounting. The main contribution of the approach used here is that the function $g$ can be determined explicitly by solving an integral equation - thanks to the representation of strictly positive subharmonic functions established in Sections \ref{s:IE} and \ref{s:existence} - whereas \cite{BL} and \cite{dayanikRD} only give an abstract definition in terms of the expectation of a multiplicative functional. 

The outline of the paper is as follows: Section \ref{s:prelim} introduces the set up and basic terminology that will be used in the paper. Section \ref{s:IE} gives a complete characterisation  and uniqueness of solutions of the integral equations that are solved by semi-bounded subharmonic functions. Section \ref{s:existence} establishes the existence of solutions for the integral equations of Section \ref{s:IE} and contains the representation result for general strictly positive subharmonic functions. The path transformations via subharmonic functions and their associated continuous additive functionals are studied in Section \ref{s:transform}. Finally, the theory is illustrated via some examples in Section \ref{s:examples} and applied to solve an optimal stopping problem in Section \ref{s:OS}.
\section{Preliminaries} \label{s:prelim}
Let $X=(\Om, \cF, \cF_t, \theta_t, P^x)$ be a regular  diffusion on $\bfE:=(l,r)$, where $ -\infty \leq l <r \leq \infty$, and $\mathcal{E}^u$ stands for the $\sigma$-algebra of universally measurable subsets of $\bfE$. If any of the boundaries are reached in finite time, the process is killed and sent to the cemetery state $\Delta$. As usual, $P^x$ is the law of the process initiated at point $x$ at $t=0$ and $\zeta$ is its lifetime, i.e. $\zeta :=\inf\{t>0:X_{t} =\Delta \}$. The transition semigroup of $X$ will be given by the kernels $(P_t)_{t \geq 0}$ on $(\bfE, \cE^u)$ and $(\theta_t)_{t \geq 0}$ is the shift operator.  The filtration $(\cF_t)_{t \geq 0}$ will denote the universal completion of the natural filtration of $X$, $\cF:=\vee_{t\geq 0}\cF_t$ and $\cF^u$ is the $\sigma$-algebra generated by the maps $f(X_t)$ with $t\geq 0$ and $f$ universally measurable\footnote{The reader is referred to Chapter 1 of  \cite{GTMP}  for the details.}.  Since $X$ is strong Markov by definition, $(\cF_t)_{t \geq 0}$ is right continuous (cf. Theorem 4 in Section 2.3 in \cite{ChungWalsh}). 

For $y \in (l,r)$ the stopping time  $T_y:=\inf\{t> 0: X_t=y\}$, where the infimum of an empty set  equals $\zeta$ by convention, is the first hitting time of $y$. Likewise $T_{ab}$ will denote the exit time from the interval $(a,b)$.  One can extend the notion of `hitting time' to each of the boundary points. To this end the random variable $T_{\ell}:\Om\to [0,\infty]$ is defined by
\be \label{e:Tell}
T_{\ell}(\om):=\left\{\ba{ll}
\zeta, & \mbox{ if } X_{\zeta-} \mbox{ exists and equals } \ell;\\
\infty, & \mbox{ otherwise;}\ea\right.\ee
and $T_r$ is defined similarly. $T_{\ell}$ and $T_r$ can be interpreted as the first hitting times of $\ell$ and $r$ although $X$ never equals them in the strict sense.

Such a one-dimensional diffusion is completely characterised by its  strictly increasing and continuous scale function $s$, speed measure $m$, and killing measure $k$.  The reader is referred to Chapter II of \cite{BorSal} for a concise treatment of these characteristics. In particular if the killing measure is null the infinitesimal generator $\cA$ of the diffusion is given by $\cA=\frac{d}{dm}\frac{d}{ds}$. 

\begin{remark}
It is worth emphasising here that no assumption of absolute continuity with respect to the Lebesgue measure  is made for the scale function or the speed measure. That is, $X$ is not necessarily the solution of a stochastic differential equation. A notable example is {\em the skew Brownian motion} (see \cite{skewBM81}).
\end{remark}

The concept of a {\em continuous additive functional} will be playing a key role throughout the paper.
\begin{definition}
	\label{d:caf} A family $A=(A_t)_{t \geq 0}$ of functions from $\Omega$ to $[0,\infty]$ is called a {\em continuous additive functional} of $X$ if
	\begin{itemize}
		\item[i)] Almost surely the mapping $t \mapsto A_t$ is nondecreasing, (finite) continuous on $[0,\zeta)$, and $A_t = A_{\zeta-}$ for $t \geq \zeta$.
		\item[ii)] $A_t \in \cF_t$ for each $t\geq 0$.
		\item[iii)] For each $t$ and $s$ $A_{t+s}=A_t + A_s \circ \theta_t$, a.s..
	\end{itemize}
\end{definition}
To each CAF $A$ one can associate a {\em Revuz measure} $\mu_A$ defined on the Borel subsets of $\elr$  by
\be \label{d:Revmes}
\int_{\elr}f(y)\mu_A(dy)=\lim_{t \rar 0}t^{-1}E^m[\int_0^tf(X_s)dA_s],
\ee
where $f$ is a non-negative Borel function. It must be noted that the Revuz measure depends on the choice of the speed measure. Moreover, in this one-dimensional setting $\mu_A$ will be a Radon measure\footnote{This is proved when $X$ is a Brownian motion in Proposition X.2.7 in \cite{RY}. However, the proof extends verbatim to all regular linear diffusions since the right continuity of the mapping $t\mapsto f(X_t)$ is equivalent to $g$ being finely continuous (cf. Theorem II.4.8 in \cite{BG}). Moreover, the fine topology induced by $X$ coincides with the standard metric topology on $\elr$ (see, e.g., Exercise 10.22 in \cite{GTMP}).}. 

One possible use of continuous additive functionals is the construction of a diffusion with non-zero killing measure from a diffusion with the same scale and speed but no killing a described in Paragraph 22 of Chapter II in \cite{BorSal}. For this reason and given the nature of questions addressed in this paper the following will be assumed throughout:
\begin{assumption} \label{a:nokill}
	The killing measure $k\equiv 0$. That is, there is no killing in the interior of the state space.
\end{assumption}
Under Assumption \ref{a:nokill} the potential density with respect to $m$ of a transient diffusion is given by
\[
u(x,y)= \lim_{a \rar \ell}\lim_{b \rar r} \frac{(s(x\wedge y)-s(a))(s(b)-s(x \vee y))}{s(b)-s(a)}, \qquad x, y \mbox{ in } \elr.
\]
In this case (see, e.g., \cite{Revuz70}) for any non-negative Borel function $f$
\be \label{e:potentialA}
E^x\left[\int_0^{\zeta}f(X_t)dA_t\right]=\int_{\ell}^{r}u(x,y)f(y)\mu_A(dy).
\ee
Moreover, the finiteness of $A_{\zeta}$, or equivalently $A_{\infty}$, is completely determined in terms of $s$ and $\mu_A$. The following is a direct consequence of Lemma A1.7 in \cite{AS98} (see \cite{MU} for an analogous result and a different technique of proof in case of $dA_t =f(X_t)dt$ for some non-negative measurable $f$).
\begin{theorem}\label{t:Afinite} Let $A$ be a CAF of $X$ with Revuz measure $\mu_A$. 
	\begin{enumerate}
		\item If $X$ is recurrent and $\mu_A(\bfE)>0$, then $A_{\zeta}=\infty$, a.s..
		\item If $s(\ell)>-\infty$, then on $[X_{\zeta-}=\ell]$, $A_{\zeta}=\infty$  a.s.  or $A_{\zeta}<\infty$  a.s. whether 
		\[
		\int_{\ell}^c(s(x)-s(\ell))\mu_A(dx)  
		\]
		is infinite or not for some $c \in \elr$.
		\item If $s(r)<\infty$, then on $[X_{\zeta-}=r]$ $A_{\zeta}=\infty$  a.s.  or $A_{\zeta}<\infty$  a.s.  whether 
		\[
		\int_c^r(s(r)-s(x))\mu_A(dx)
		\]
		is infinite or not for some $c \in \elr$.
	\end{enumerate}
\end{theorem}
\section{Integral equations for positive subharmonic functions}\label{s:IE}
\begin{definition}
A Borel function $g:\bfE \to \bbR$ is {\em subharmonic} if  for any  $x \in (\ell,r)$  $X^{T_{ab}}$ is a uniformly integrable $P^x$-submartingale whenever $\ell<a<b<r$. The class of non-negative subharmonic functions is denoted by $\cS$. Similarly, $\cS^+$ will be the set of elements of $\cS$ that are strictly positive on $(\ell,r)$. 
\end{definition}
Since $E^x[g(X_{T_{ab}})]=g(a)P^x(T_a<T_b)+g(b)P^x(T_b<T_a)=g(a)\frac{s(b)-s(x)}{s(b)-s(a)}+g(b)\frac{s(x)-s(a)}{s(b)-s(a)}$, one immediately deduces that $g$ must be a convex function of $s$ on the open interval $(\ell,r)$ to be subharmonic, which in particular entails that $g$ is absolutely continuous on $(\ell,r)$.
\begin{remark}
	In the sequel whenever a convex function is considered on some open interval $(a,b)$ it will be automatically extended to $[a,b]$ by continuity.
\end{remark}
By definition given any subharmonic function $g$, $g(X)$ is a local submartingale and, therefore, there exists a unique CAF $B$ with $B_0=0$ such that $g(X) -B$ is a $P^x$-local martingale for any $x \in (\ell,r)$ by a Markovian version of the Doob-Meyer decomposition (see Theorem 51.7 in \cite{GTMP}).  If $g$ is further supposed to be in $\cS^+$, then a simple integration by parts argument yields a unique $A$ such that $g(X)\exp(-A)$ is a local martingale, where $A$ is an adapted,  continuous and increasing process with $A_0=0$. Clearly, this $A$ is defined by its initial condition and $g(X_t)dA_t=dB_t$. 

The above argument gives a multiplicative decomposition for $g\in \cS^+$ as a product of a local martingale and an increasing process. The strict positivity is essential and the following summarises the above discussion.
\begin{theorem} \label{t:submvanish}
	Let $g \in \cS$, then $g$ is $s$-convex. Suppose further that $g$ is not identically $0$. There exists a CAF $A$  such that $g(X)\exp(-A)$ is a $P^x$-local martingale for every $ x\in (\ell,r)$ if and only if $g$ never vanishes on $(\ell,r)$.
\end{theorem}
\begin{proof}
	What remains to be proven is the implication that the existence of an $A$ with above properties implies strict positivity of $g$. To this end suppose the closed set $Z:=\{x\in (\ell,r):g(x)=0\}$ is not empty. By the regularity of $X$ $P^x(T_Z<\zeta)>0$ for any $x \in (\ell,r)$. Since $g(X)\exp(-A)$ is a supermartingale being a non-negative local martingale, it will remain zero on $[T_Z, \zeta)$, which in turn implies $X$ does not leave the set $Z$ on $[T_Z, \zeta)$ since $\exp(-A_t)> 0$ on $[t<\zeta]$. One then deduces via the strong Markov property of $X$ that  $P^z(T_y <\zeta)=0$ for any $z \in Z$ and $y \in Z^c$. However, this contradicts the regularity of $X$.
\end{proof}

Although the strict positivity is essential for the above argument, one does not lose any generality by considering only functions in $\cS^+$ since $1+g \in \cS^+$ for any $g \in \cS$.

Also note that one can turn the above arguments backwards and show the existence of an $g \in \cS^+$ such that $g(X)\exp(-A)$ is a local martingale for a given $A$. Consequently, the local submartingale $g(X)$ has a multiplicative decomposition as a product of a local martingale and an increasing process. Such multiplicative decompositions in the context of Markov process goes back to the work of \Ito { and}  Watanabe \cite{IWmult} who studied  multiplicative decompositions of supermartingales and their use in the study of subprocesses. This historical note motivates the following definition. 
\begin{definition}
	$(g,A)$ is called an {\em \Ito-Watanabe pair} if $A$ is a CAF, $g\in \cS^+$,and $g(X)\exp(-A)$ is a $P^x$-local martingale for all $x \in \elr$.
\end{definition}

 The main purpose of this section is the construction of non-negative subharmonic functions appearing in \Ito-Watanabe pairs given a CAF  of $X$. Note that in general there is no uniqueness for such $g$. For instance, if $A_t=t$, then the increasing and decreasing solutions of $Ag =g$ belong to $\cS^+$ and $g(X)\exp(-A)$ are local martingales.   

\begin{definition}
A function $g$ is said to be {\em uniformly integrable near $\ell$ (resp. $r$)} if the family $\{g(X^{T_b}_{\tau}):\tau \mbox{ is a stopping time}\}$ is $P^x$-uniformly integrable for any $x <b$ (resp. $x>b$).  $g$ is said to be {\em semi-uniformly integrable} if it is uniformly integrable near $\ell$ or $r$. 
\end{definition}
\begin{proposition} \label{p:semiui}
	Consider an \Ito-Watanabe pair $(g,A)$, where $g$ is semi-uniformly integrable, and let $\mu_A$ be the Revuz measure associated to $A$. Then,  the following hold:
	\begin{enumerate}
		\item If $g$ is uniformly integrable near $\ell$, for any $b \in \elr$ and $x <b$ 
		\be \label{e:huil}
		E^x\int_0^{T_{b}}g(X_t)dA_t=\int_{\ell}^b\lim_{a \rar \ell}\frac{\left(s(x\wedge y)-s(a)\right)\left(s(b)-s(x\vee y\right)}{s(b)-s(a)}g(y)\mu_A(dy)<\infty. 
		\ee
		Moreover, $g(\ell)<\infty$ if $s(\ell)>-\infty$ and 
		\be \label{e:huill}
		\lim_{x \rar \ell} \frac{g(x)}{s(x)}=0 \mbox{ if } s(\ell)=-\infty.
		\ee
		Furthermore, given $s(\ell)>-\infty$ and $\int_{\ell}^c g(y)\mu_A(dy)<\infty$ for some $c$,
		\be \label{g:uilder}
		\frac{d^+g(\ell)}{ds}=\frac{g(b)-g(\ell)}{s(b)-s(\ell)}-\int_{\ell}^b\frac{s(b)-s(z)}{s(b)-s(\ell)}g(z)\mu_A(dz),
		\ee
		for any $b \in \elr$.
			\item  If $g$ is uniformly integrable near $r$, for any $b \in \elr$ and $x >b$
		 \be \label{e:huir}
		E^x\int_0^{T_{b}}g(X_t)dA_t=\int_{b}^r\lim_{c \rar r}\frac{\left(s(x\wedge y)-s(b)\right)\left(s(c)-s(x\vee y\right)}{s(c)-s(b)}g(y)\mu_A(dy)<\infty. 
		\ee
			Moreover,  $g(r)<\infty$ if $s(r)<\infty$ and 
		\be \label{e:huirl}
		\lim_{x \rar r} \frac{g(x)}{s(x)}=0 \mbox{ if } s(r)=\infty.
		\ee
		Furthermore, given $s(r)<\infty$ and $\int_c^r g(y)\mu_A(dy)<\infty$ for some $c$,
		\be \label{g:uirder}
		\frac{d^-g(r)}{ds}=\frac{g(r)-g(b)}{s(r)-s(b)}+\int_{b}^r\frac{s(z)-s(b)}{s(r)-s(b)}g(z)\mu_A(dz),
		\ee
		for any $b \in \elr$.
	\end{enumerate}
\end{proposition}
\begin{proof}
	A straightforward integration by parts yields $N:=g(X)-\int_0^{\cdot}g(X_s)dA_s$ is a $P^x$-local martingale for any $x\in \elr$.  First, pick $a,x$ and $b$ so that $\ell<a<x <b$ and observe that $N$ is $P^x$-uniformly integrable when stopped at $T_{ab}$ since $g$ is bounded on $(a,b)$ and $E^x[A_{T_{ab}}]<\infty$ in view of (\ref{e:potentialA}) and that $\mu_A$ is a Radon measure. Thus,
	\be \label{e:Hab}
	g(a)\frac{s(b)-s(x)}{s(b)-s(a)}+g(b)\frac{s(x)-s(a)}{s(b)-s(a)}=E^x g(X_{T_{ab}})=g(x)+E^x\int_0^{T_{ab}}g(X_t)dA_t.
	\ee
	Next suppose that $g$ is uniformly integrable near $\ell$. Then, by letting $a \rar \ell$, the right hand side of the above increases to
	\[
	g(x)+E^x\int_0^{T_{b}}g(X_t)dA_t
	\]
	while the left hand side converges to $	E^x g(X_{T_{b}})<\infty$ by the assumed form of semi-uniform integrability. This establishes (\ref{e:huil}) in view of (\ref{e:potentialA}) since the potential kernel $u(b;\cdot,\cdot)$ of $X$ killed when exiting $(l,b)$ is given by
	\be \label{e:ulb}
	u(b;x,y)=\lim_{a \rar \ell}\frac{\left(s(x\wedge y)-s(a)\right)\left(s(b)-s(x\vee y\right)}{s(b)-s(a)}.
	\ee
	Finiteness of $E^x g(X_{T_{b}})$ implies that of 
	\[
	\limsup_{a \rar  \ell} g(a)\frac{s(b)-s(x)}{s(b)-s(a)},
	\]
	which in turn yields the finiteness of $g(\ell)$ when $s(\ell)>-\infty$.
	
	Moreover, since $g(X_{T_b})=g(b)$ when $s(\ell)=-\infty$, one obtains
	\[
	0=\lim_{a \rar \ell}g(a)\frac{s(b)-s(x)}{s(b)-s(a)},
	\]
	which proves (\ref{e:huill}).
	
	To find the right derivative of $g$ at $\ell$, first observe that (\ref{e:huil}) yields for any $x<b$
	\[
	\begin{split}
		E^x[g(X_{T_b})]=&g(\ell)\frac{s(b)-s(x)}{s(b)-s(\ell)}+ g(b)\frac{s(x)-s(\ell)}{s(b)-s(\ell)}\\
		=& g(x) +\int_{\ell}^b  \frac{(s(x\wedge z)-s(\ell))(s(b)-s(x \vee z)}{s(b)-s(\ell)})g(z)\mu_A(dz).
	\end{split}
	\]
	Consequently,
	\[
	\frac{g(x)-g(\ell)}{s(x)-s(\ell)}=\frac{g(b)-g(\ell)}{s(b)-s(\ell)} -\int_{\ell}^x\frac{(s(z)-s(\ell))(s(b)-s(x))}{(s(x)-s(\ell))(s(b)-s(\ell))}g(z)\mu_A(dz)-\int_{x+}^b \frac{s(b)-s(z)}{s(b)-s(\ell)}g(z)\mu_A(dz).
	\]
	Taking limits as $x \rar \ell$ and utilising $\int_{\ell}^c g(y)\mu_A(dy)<\infty$ for some $c$ yield the desired result. 
	
	Similar arguments apply when $g$ is uniformly integrable near $r$.
\end{proof}
\begin{theorem}\label{t:main1}
	Let $g$ be a  Borel measurable function on $\bfE$ and $A$ a  CAF with Revuz measure $\mu_A$. Then the following are equivalent:
	\begin{enumerate}[label=(\arabic*),ref=(\arabic*)]
		\item \label{t:eq1} $(g,A)$ is an \Ito-Watanabe pair, $g$ is uniformly integrable near $\ell$ (resp. near $r$) and $s(\ell)=-\infty$ (resp. $s(r)=\infty$).
		\item \label{t:eq2} $s(\ell)=-\infty$ (resp. $s(r)=\infty$), $g$ solves the integral equation\footnote{Any solution is implicitly assumed to be integrable in the sense that $\int_{\ell}^r |v_c(x,y)g(y)|\mu_A(dy)<\infty$ for all $x \in \elr$.}
		\be \label{e:IEh}
		g(x)=g(c)+ \int_{\ell}^r v_c(x,y)g(y)\mu_A(dy),
		\ee
		where 
		\[
		v_c(x,y)= s(x\vee y)-s(c\vee y)\; \left(\mbox{resp. }s(c\wedge y)-s(x\wedge y)\right),
		\]
		and $g(x)>0$ for some $x \in \elr$.
	\end{enumerate}
\end{theorem}
The following key lemma, whose proof is delegated to the Appendix, will be useful in proving the above theorem and some subsequent results.
\begin{lemma}\label{l:mainR}
	Let $g$ be a solution of 
	\be \label{e:genelIE}
	g(x)=g(c)+ \kappa (s(x)-s(c))+ \int_{\ell}^r v_c(x,y)g(y)\mu_A(dy),
	\ee
	 and define $O^+:=\{x\in \elr:g(x)>0\}$  and $O^-:=\{x\in \elr:g(x)<0\}$. Then, following statements are valid:
	 \begin{enumerate}
	 	\item For any $\ell <a<x<b<r$
	 	\be 
	 	E^x[g(X_{T_{ab}})]=g(x)+ \int_a^b \frac{(s(x\wedge y)-s(a))(s(b)-s(x\vee y))}{s(b)-s(a)}g(y)\mu_A(dy)\label{e:EhabR}
	 	\ee
	 	\item $g$ is $s$-convex on $O^+$ and $s$-concave on $O^-$.
	 	\item If  $v_c(x,y)= s(x\vee y)-s(c\vee y)\, \left(\mbox{resp. }s(c\wedge y)-s(x\wedge y)\right)$ and  $s(\ell)=-\infty$ (resp. $s(r)=\infty$), then $\kappa=-\lim_{x\rar \ell}\frac{g(x)}{s(b)-s(x)}$ for any $b \in \elr$. In particular, $\kappa=0$ if $g$ is uniformly integrable near $\ell$ (resp. $r$).
	 	\item  If $\kappa= 0$, $g$ does not change sign in $\elr$.
	 	\item If $	v_c(x,y)= s(x\vee y)-s(c\vee y)\; \left(\mbox{resp. }s(c\wedge y)-s(x\wedge y)\right)$ and $g\geq 0$, $g-\kappa s$ is increasing (resp. decreasing).
	 	\item $g$ is differentiable with respect to $s$ from  left and  right with following derivatives:
	 	\be \label{e:gsder}
	 	\begin{split}
	 		\frac{d^+g(x)}{ds}&=\left\{\ba{ll}\kappa +\int_{\ell}^x g(y)\mu_A(dy), & \mbox{if }v_c(x,y)= s(x\vee y)-s(c\vee y); \\
	 		\kappa - \int_{x+}^r g(y)\mu_A(dy), & \mbox{if }v_c(x,y)= s(c\wedge y)-s(x\wedge y).
	 		\ea\right.
	 		\\
	 		\frac{d^-g(x)}{ds}&=\left\{\ba{ll}\kappa +\int_{\ell}^{x-} g(y)\mu_A(dy), & \mbox{if }v_c(x,y)= s(x\vee y)-s(c\vee y); \\
	 		\kappa - \int_{x}^r g(y)\mu_A(dy), & \mbox{if }v_c(x,y)= s(c\wedge y)-s(x\wedge y).
	 		\ea  \right.
	 		\end{split}
 		\ee
 		Consequently, $g$ is differentiable with respect to $s$ at $x$ if $\mu_A(\{x\})=0$ or $g(x)=0$. Moreover, $\frac{d^+g(\ell)}{ds}$ (resp. $\frac{d^-g(r)}{ds}$) exists and satisfies the above formula whenever $g(\ell)<\infty$ (resp. $g(r)<\infty$).
	 \end{enumerate} 
\end{lemma}

\begin{remark} \label{r:gsign}
	If $\kappa\neq 0$, $g$ can change sign. Indeed, suppose $\elr=(-1,1), \, \mu_A(dy)=dy$, and $s(x)=x$. Then, $g(x)=\sinh(x)$ solves
	\[
	g(x)=\cosh(-1) \sinh(x) +\int_{\ell}^x(x-y^+)g(y)\mu_A(dy).
	\]
	Clearly, this is linked to a Brownian motion on $(-1,1)$. $\sinh(B_t)\exp(-\frac{t}{2})$ is a local martingale that hits $0$ infinitely many times. 
\end{remark}
\begin{proof}[Proof of Theorem \ref{t:main1}]
	
	\Implies{t:eq1}{t:eq2}: As in the proof of Proposition \ref{p:semiui} consider $\ell <a <x <b <r$ and assume further that $b>c$. Suppose that $g$ is uniformly integrable near $\ell$, which in particular implies (\ref{e:huill}). Then, (\ref{e:Hab}), (\ref{e:huil}) and (\ref{e:huill})   yield  
	\[
	g(x)=g(b)-\int_{\ell}^b u(b;x,y)g(y)\mu_A(dy),
	\]
where 
\[
u(b;x,y)=s(b)-s(x\vee y)
\]
in view of  (\ref{e:ulb}). 	That is,
\[
g(x)=g(c)+ \int_{\ell}^b\left(u(b;c,y)-u(b;x,y)\right)g(y)\mu_A(dy)=g(c)+ \int_{\ell}^b\left(s(x\vee y) -s(c\vee y)\right)g(y)\mu_A(dy).
\]
Since
\[
\int_{x \vee c}^b\left(s(x\vee y) -s(c\vee y)\right)g(y)\mu_A(dy)=0,
\]
the claim follows by the arbitrariness of $b$. The case of uniform integrability near $r$ is handled similarly. 

\Implies{t:eq2}{t:eq1} The proof will be given for $v_c(x,y)=s(x\vee y)-s(c\vee y)$. The other case can be done similarly. 

 It follows from Lemma \ref{l:mainR} that $g$ is non-negative and $s$-convex on $\elr$. Thus, $g$ is subharmonic and there exists a CAF $B$ by Theorem 51.7 in \cite{GTMP} that $g(X)-B$ is a $P^x$-local martingale for any $x \in \elr$. In particular, for any $\ell <a<x<b<r$,
\[
E^x[g(X_{T_{ab}})]=g(x)+ E^x[B_{T_{ab}}]=g(x)+ \int_a^b \frac{(s(x\wedge y)-s(a))(s(b)-s(x\vee y))}{s(b)-s(a)}\mu_B(dy)
\]
due to (\ref{e:potentialA}), where $\mu_B$ is the Revuz measure associated with $B$. On the other hand, (\ref{e:EhabR}) yields
\[
E^x[g(X_{T_{ab}})]=g(x)+  \int_a^b \frac{(s(x\wedge y)-s(a))(s(b)-s(x\vee y))}{s(b)-s(a)}g(y)\mu_A(dy).
\]
Since 
\[
\int_a^b \frac{(s(x\wedge y)-s(a))(s(b)-s(x\vee y))}{s(b)-s(a)}g(y)\mu_A(dy)=E^x\int_0^{T_{ab}} g(X_t)dA_t,
\]
one deduces easily that $E^x[B_{T_{ab}}]=E^x\int_0^{T_{ab}} g(X_t)dA_t$ for all $a<x<b$. That is, the potentials of $g \cdot A$ and $B$ coincide when $X$ is killed at $T_{ab}$, which in turn leads to the fact that $B$ and $g \cdot A$  are indistinguishable by Theorem IV.2.13 in \cite{BG} since $a$ and $b$ are arbitrary. Thus, $g(X)-\int_0^{\cdot}g(X_t)dA_t$ is a local martingale. A simple integration by parts and the fact that $g$ is bounded on the compact intervals of $\elr$ show that $g(X)\exp(-A)$ is a local martingale. Since $g$ is not identically $0$, $(g,A)$ is an \Ito-Watanabe pair in view of Theorem \ref{t:submvanish}.

Uniform integrability near $\ell$ is obvious since $g$ is bounded on $(\ell, b)$ for any $b<r$ in view of Lemma \ref{l:mainR} and the fact that $g\geq 0$. 
\end{proof}
Since $-s(\ell)=s(r)=\infty$, when $X$ is recurrent, the following corollary is immediate.
\begin{corollary}\label{c:mainR}
	Let $g$ be a  Borel measurable function on $\bfE$ and $A$ a  CAF with Revuz measure $\mu_A$. Suppose further that $X$ is recurrent. Then the following are equivalent:
	\begin{enumerate}
		\item $(g,A)$ is an \Ito-Watanabe pair, $g$ is uniformly integrable near $\ell$ (resp. near $r$).
		\item  $g$ solves the integral equation
		\be \label{e:IEhR}
		g(x)=g(c)+ \int_{\ell}^r v_c(x,y)g(y)\mu_A(dy),
		\ee
		where 
		\[
		v_c(x,y)= s(x\vee y)-s(c\vee y)\; \left(\mbox{resp. }s(c\wedge y)-s(x\wedge y)\right),
		\]
		and $g(x)>0$ for some $x \in \elr$.
	\end{enumerate}
\end{corollary}
That $s(\ell)=-\infty$ is not a necessary condition for a subharmonic function that is uniformly integrable near $\ell$ to satisfy (\ref{e:IEh}). However, the situation is rather delicate since non-negative {\em harmonic} functions can appear in the decomposition of subharmonic functions. The following special case will  be instrumental for the integral equations satisfied by general semi-uniformly integrable $g \in \cS^+$. Note that if $g(\ell)=0$, $g$ is uniformly integrable near $\ell$. 
\begin{proposition} \label{p:hhp0}
		Let $g$ be a  Borel measurable function on $\bfE$ with $g(\ell)=0$ (resp. $g(r)=0)$ and $A$ a CAF with Revuz measure $\mu_A$.  Then the following are equivalent:
		\begin{enumerate}[label=(\arabic*),ref=(\arabic*)]
			\item \label{p:eq1} $(g,A)$ is an \Ito-Watanabe pair,  $s(\ell)>-\infty$ (resp. $s(r)<\infty$) and $\frac{d^+g(\ell+)}{ds}=0$ (resp. $\frac{d^-g(r-)}{ds}=0$).
			\item \label{p:eq2} $g$ solves the integral equation
			\be \label{e:IEg_tr_fin}
			g(x)=g(c)+ \int_{\ell}^r v_c(x,y)g(y)\mu_A(dy),
			\ee
			where 
			\[
			v_c(x,y)= s(x\vee y)-s(c\vee y)\; \left(\mbox{resp. }s(c\wedge y)-s(x\wedge y)\right),
			\]
			and $g(x)>0$ for some $x \in \elr$.
		\end{enumerate}
	\end{proposition}
\begin{proof} Suppose that $g(\ell)=0$ and $s(\ell)$ is finite. The other case can be handled similarly. 
	
	\Implies{p:eq1}{p:eq2}: First observe that $g$ is $s$-convex and therefore increasing as $g(\ell)=0$. Idea of the proof is to pass to an absolutely continuous measure via an $h$-transform so that the scale function of the diffusion becomes infinite and $\frac{g}{h}$ remains u.i. near $\ell$. To this end consider $h(x)=s(x)-s(\ell)$ and let $P^h$ denote the law of the $h$-transformed process defined by $h(x)P^{h,x}(C)= E^x[h(X_t)\chf_{\{t<\zeta\}};C]$ for any $C \in \cF_t$ (see Section 62 of \cite{GTMP} or Paragraphs 31 and 32 in Chap. II of \cite{BorSal}). Then
	\[
	s^h(x):=\int_c^x \frac{1}{(s(z)-s(\ell))^2}ds(z)=\frac{1}{s(c)-s(\ell)}- \frac{1}{s(x)-s(\ell)}
	\]
	is a scale function under $P^h$ (see Paragraph 31 in Chap. II of \cite{BorSal}). Clearly, $s^h(\ell)=-\infty$. Moreover, $s^h(r)<\infty$. That is, $X$ converges to $r$ with probability $1$ under $P^h$. 
	
	Let us next see that $\frac{g}{h}$ is  u.i near $\ell$ under $P^h$. Indeed, for any $\ell<a<x<b<r$,
	\[
	\lim_{a \rar \ell}h(x)E^{h,x}\left[\frac{g(X_{T_{ab}})}{h(X_{T_{ab}})}\right]=\lim_{a \rar \ell}E^x[g(X_{T_{ab}})]=E^x[g(X_{T_{b}})]=g(b)P^x(T_b<\zeta),
	\]
	where the last equality follows from the hypothesis that $g(l)=0$. However, 
	\[
	P^x(T_b<\zeta)= \frac{E^x[h(X_{T_b}); T_b<\zeta]}{h(b)}=\frac{h(x)}{h(b)}P^{h,x}(T_b<\zeta)=\frac{h(x)}{h(b)}
	\]
	since under $P^h$ $X$ converges to $r$ and $x<b$. Therefore, 
	\[
	\lim_{a \rar \ell}E^{h,x}\left[\frac{g(X_{T_{ab}})}{h(X_{T_{ab}})}\right]=\frac{g(b)}{h(b)},
	\]
	which establishes the desired semi-uniform integrability. 
	
	Moreover, due to the above absolute continuity relationship, $\frac{g(X)}{h(X)}\exp(-A)$ is a $P^{h,x}$-local martingale. Therefore, the conditions of Theorem \ref{t:main1} are satisfied and one has
	\[
	\frac{g(x)}{h(x)}=\frac{g(c)}{h(c)}+ \int_{\ell}^r (s^h(x \vee y)-s^h(c \vee y)) \frac{g(y)}{h(y)} \mu^h_A(dy),
	\]
	where $c\in \elr$ is arbitrary and $\mu^h_A$ is the Revuz measure of $A$ after the $h$-transform. Since the speed measure of the $h$-transformed process $m^h$ is given by $m^h(dy)= h^2(y)m(dy)$ (see, once more, Paragraph 31 in Chap. II of \cite{BorSal}), it follows that  $\mu^h_A(dy)=h^2(y) \mu_A(dy)$. Therefore, assuming without loss of generality that $x>c$ and using the explicit form of $s^h$, one obtains
	\be \label{e:goverh}
	\begin{split}
\frac{g(x)}{s(x)-s(\ell)}&=\frac{g(c)}{s(c)-s(\ell)}
+\int_{\ell}^x \frac{s(x)-s(c \vee y)}{(s(x)-s(\ell))(s(c\vee y)-s(\ell))}g(y)(s(y)-s(\ell))\mu_A(dy)\\
&=\frac{g(c)}{s(c)-s(\ell)}
+\int_{\ell}^c \frac{(s(x)-s(c))(s(y)-s(\ell))}{(s(x)-s(\ell))(s(c)-s(\ell))}g(y)\mu_A(dy)\\
&+\int_c^x \frac{s(x)-s(y)}{s(x)-s(\ell)}g(y)\mu_A(dy)
	\end{split}
	\ee
If one considers the limit of the right hand side of the above, $\lim_{c \rar \ell}\frac{g(c)}{s(c)-s(\ell)}=0$ since $g(\ell)= \frac{d^+g(\ell+)}{ds}=0$ and $\frac{g(c)}{s(c)-s(\ell)}\leq  \frac{d^+g(c)}{ds}$.  To understand the remaining limit consider $E^{\frac{c+\ell}{2}}[g(X_{T_c}) ]$. Since $g$ is uniformly integrable near $\ell$ and $g(\ell)=0$,
\bean
g(c)\frac{s(\frac{c+\ell}{2})-s(\ell)}{s(c)-s(\ell)}&=& g\left(\frac{c+\ell}{2}\right)+\int_{\ell}^c\frac{\left(s(\frac{c+\ell}{2}\wedge y)-s(\ell)\right)\left(s(c)-s(\frac{c+\ell}{2}\vee y)\right)}{s(c)-s(\ell)}g(y)\mu_A(dy)\\
&\geq&g\left(\frac{c+\ell}{2}\right)+\int_{\ell}^{\frac{c+\ell}{2}}\frac{\left(s(y)-s(\ell)\right)\left(s(c)-s(\frac{c+\ell}{2})\right)}{s(c)-s(\ell)}g(y)\mu_A(dy).
\eean

Thus,
\bean
\int_{\ell}^{\frac{c+\ell}{2}}\frac{\left(s(y)-s(\ell)\right)}{s(c)-s(\ell)}g(y)\mu_A(dy)&\leq &-\frac{g(c)}{s(c)-s(\ell)}+ \frac{g(c)-g(\frac{c+\ell}{2})}{s(c)-s(\frac{c+\ell}{2})}\\
&\leq&-\frac{g(c)}{s(c)-s(\ell)}+ \frac{d^+g(c)}{ds},
\eean
where the last inequality follows from the fact that $g$ is $s$-convex as well as increasing. The right hand side converges to $0$ as $c \rar \ell$ by the hypothesis. Therefore,
\[
\lim_{c \rar \ell}\int_{\ell}^c \frac{(s(x)-s(c))(s(y)-s(\ell))}{(s(x)-s(\ell))(s(c)-s(\ell))}g(y)\mu_A(dy)=0
\]
and, consequently,
\[
g(x)= \int_{\ell}^x \frac{s(x)-s(y)}{s(x)-s(\ell)}g(y)\mu_A(dy).
\]
This implies (\ref{e:IEg_tr_fin}).

\Implies{p:eq2}{p:eq1}: This follows exactly the same lines of the proof of the corresponding statement in Theorem \ref{t:main1}. That $\frac{d^+g(\ell+)}{ds}=0$ is a consequence of (\ref{e:gsder}) since $\mu_A$ does not charge $\{\ell\}$.
\end{proof}

Note that if $s(\ell)>-\infty$ and $g$ is $s$-convex with $0<g(\ell)<\infty$ and $\frac{d^+g(\ell+)}{ds}<\infty$, one can consider $\tilde{g}(x):=g(x)-g(\ell)-\frac{d^+g(\ell+)}{ds}(s(x)-s(\ell))$. Then, $\tilde{g}$ is $s$-convex with $\tilde{g}(\ell)=\frac{d^+g(\ell+)}{ds}=0$. This observation leads to the following theorem, whose proof being similar to that of Theorem \ref{t:main1} is delegated to the Appendix.
\begin{theorem}\label{t:main2}
	Let $g$ be a  Borel measurable function on $\bfE$ and $A$ a CAF with Revuz measure $\mu_A$. Then the following are equivalent:
	\begin{enumerate}[label=(\arabic*),ref=(\arabic*)]
		\item \label{t:main:eq1} $(g,A)$ is an \Ito-Watanabe pair and $g$ is uniformly integrable near $\ell$ (resp. near $r$) with $\frac{d^+g(\ell+)}{ds}<\infty$ (resp.  $\frac{d^-g(r-)}{ds}<\infty$) and $s(\ell)>-\infty$ (resp.  $s(r)<\infty$).
		\item \label{t:main:eq2} $s(\ell)>-\infty$ (resp. $s(r)<\infty$),  $g$ solves the integral equation\footnote{Any solution is implicitly assumed to be integrable in the sense that $\int_{\ell}^r |v_c(x,y)g(y)|\mu_A(dy)<\infty$ for all $x \in \elr$.}
		\be \label{e:IEh_general}
		g(x)=g(c)+ \kappa (s(x)-s(c))+ \int_{\ell}^r v_c(x,y)g(y)\mu_A(dy),
		\ee
		where $\kappa= \frac{d^+g(\ell+)}{ds}$ (resp. $\kappa=\frac{d^-g(r-)}{ds}$),
		\[
		v_c(x,y)= s(x\vee y)-s(c\vee y)\; \left(\mbox{resp. }s(c\wedge y)-s(x\wedge y)\right)
		\]
		and $g(x)\geq 0$ for all $x \in \elr$.
	\end{enumerate}
\end{theorem}
\begin{remark}
Note that the non-negativity assumption is needed since solutions of (\ref{e:genelIE}) can hit $0$ and change sign when $\kappa \neq 0$ in view of Remark \ref{r:gsign}.
\end{remark}

The integral equation (\ref{e:genelIE}) typically needs two independent boundary or initial conditions to admit a unique solution. Fixing the value of $g(c)$ in (\ref{e:genelIE}) handles one of these conditions. However, (\ref{e:gsder}) also shows that $\frac{d^+g(\ell)}{ds}=\kappa$ when $g$ is uniformly integrable near $\ell$ and $s(\ell)$ is finite. That is, there is a second initial boundary condition implicit in the equation and one should expect uniqueness by fixing the value of $g(c)$. 
\begin{theorem}
	\label{t:uniqueIE}
	Let $c\in \elr$ be fixed and $a \in (0,\infty)$. Then there exists at most one solution to 
	\be \label{e:ieunique}
	g(x)=a+ \kappa (s(x)-s(c)) +\int_{\ell}^r v_c(x,y)g(y)\mu_A(dy),
	\ee
	where $\mu_A$ is the Revuz measure associated with a CAF $A$ and $v_c(x,y)$ is either $s(x\vee y)-s(c\vee y)$ for all $(x,y) \in \elr \times \elr$ or $s(c\wedge y)-s(x \wedge y)$ for all $(x,y)\in \elr \times \elr$ such that $g$ is uniformly integrable near $\ell$ (resp. $r$) whenever $v_c(x,y)=s(x\vee y)-s(c\vee y)$ (resp. $v_c(x,y)=s(c\wedge y)-s(x \wedge y))$ and $s(\ell)=-\infty$ (resp. $s(r)=\infty)$.
\end{theorem}
\begin{proof}
Proof will be given when $v_c(x,y)=s(x\vee y)-s(c\vee y)$, the other case being analogous. First consider the case $s(\ell)=-\infty$.  Let $f$ and $g$ be two solutions of (\ref{e:ieunique}) that are uniformly integrable near $\ell$. Then, for any $x \in (\ell,c)$,
	\[
	f(x)-g(x)=E^x\left[(f(X_{T_c})-g(X_{T_c}))\exp(-A_{T_c})\right]=0,
	\]
	since $P^x(T_c<T_{\ell})=1$ when $s(\ell)=-\infty$. This shows that $f$ and $g$ coincide for any $x <c$. 
	
	Next consider $a<x<c<y<r$. Using the semi-uniform integrability of $f-g$, one can then conclude
	\[
	0=f(x)-g(x)=E^x\left[(f(X_{T_{ay}})-g(X_{T_{ay}}))\exp(-A_{T_{ay}})\right]=E^x\left[\chf_{[T_y<T_a]}(f(y)-g(y))\exp(-A_{T_{ay}})\right].
	\]
	Hence, $f$ and $g$ coincide on $(c,r)$, too.
	
	Now, suppose $s(\ell)>-\infty$ and $f$ and $g$ are two solutions of (\ref{e:ieunique}). Then, (\ref{e:gsder}) yields \[
	\kappa= \frac{d^+g(\ell)}{ds}= \frac{d^+f(\ell)}{ds}.
	\]
	Define $h=f-g$ and observe that $h$ satisfies
	\[
	h(x)= \int_{\ell}^rv_c(x,y) h(y)\mu_A(dy), \qquad h(c)=0.
	\]
	Then, Lemma \ref{l:mainR} shows that $h$ does not change its sign on $(\ell,c)$. Without loss of generality suppose $h\geq 0$ on $(\ell,c)$.  Another application of Lemma \ref{l:mainR} now yields $h$ is $s$-convex. Moreover,
	$\frac{d^+h(\ell)}{ds}=\frac{d^+f(\ell)}{ds}-\frac{d^+g(\ell)}{ds}=0$. However, together with the condition that $h(c)=0$, this implies $h$ must be identically $0$ on $[\ell,c]$. That is, $f$ and $g$ coincide on $[\ell,c]$. The same martingale argument above  shows that they coincide on $[\ell,r)$.
\end{proof}
The following integration-by-parts type result regarding the solutions of (\ref{e:ieunique}) will be instrumental in Section \ref{s:transform}.
\begin{theorem} \label{t:ibp}
Suppose  $g$ solves (\ref{e:ieunique}). Then
\[
dg(X_t)s(X_t)=s(X_t)dg(X_t)+g(X_t)ds(X_t)+ \frac{d^-g(X_t)}{ds}dB_t,
\]
where $B$ is a CAF whose Revuz measure is $2s(dy)$.
\end{theorem}
\begin{proof}
	First note that there exists a CAF $B^0$ such that $s^2(X)-B^0$ is a local martingale by Theorem 51.2 in \cite{GTMP}. Thus, if $\ell <a<b<r$ then
	\[
	E^x[s^2(X_{T_{ab}})]=s^2(x)+E^x[B^0_{T_{ab}}]=s^2(x)+\int_a^b u_{ab}(x,y)\mu^0(dy),
	\]
	where $\mu^0$ is the Revuz measure of $B^0$ and $u_{ab}(x,y)= \frac{(s(x\wedge y)-s(a))(s(b)-s(x \vee y))}{s(b)-s(a)}$. Then repeating the same calculations in the proof of Theorem VII.3.12 in \cite{RY}, after replacing $Af(y)m(dy)$ therein by $\mu^0(dy)$, one obtains
	\[
	\frac{ds^2(x)}{ds}-\frac{ds^2(y)}{ds}=\int_x^y \mu^0(dy).
	\]
	That is, $\mu^0(dy)= 2s(dy)$. 
	
	Moreover, for any $y \in (l,r)$, \Ito-Tanaka formula (see, e.g., Theorem 68 in Chap. IV of \cite{Pro}) in conjuntion with $d[s(X),s(X)]_t=dB^0_t$ yields
	\[
	ds(X_t) s(X_t \vee y)= s(X_t \vee y)ds(X_t)+ s(X_t) ds(X_t \vee y) +  \chf_{[X_t> y]}dB^0_t.
	\]
	Thus, if $v_c(x,y)=s(x\vee y)-s(c\vee y)$,
	\[
	dg(X_t)s(X_t)= s(X_t)dg(X_t)+ g(X_t)ds(X_t)+ \kappa dB^0_t+\int_{\ell}^{X_t-} g(y)\mu_A(dy) dB_t^0.
	\]
However, $\int_{\ell}^{x-} g(y)\mu_A(dy)=\frac{d^-g(x)}{ds}$ by (\ref{e:gsder}), which	establishes the claim. The case of $v_c(x,y)=s(c \wedge y)-s(c\wedge x)$ is treated similarly. 
	\end{proof}
\section{Existence of solutions and further properties} \label{s:existence}
In view of Lemma \ref{l:mainR}  any solution of (\ref{e:IEhR}) is monotone and bounded on either $(\ell,c)$ or $(c,r)$.  The next result finds explicit monotone and semi-bounded solutions of (\ref{e:genelIE}) that go beyond the recurrent setting.
\begin{theorem} \label{t:semiinf_s}
	Suppose that  $A$ is a CAF. Then, for any $a \in (0,\infty)$ and $c \in \elr$ the following hold:
	\begin{enumerate}
		\item \label{i:semiinf:case1} The increasing function
		\[
		g_r(x):=\left\{\ba{ll}
		aE^x[\chf_{[T_c<T_{\ell}]}\exp(-A_{T_c})],& x\leq c,\\
		\frac{a}{E^c[\chf_{[T_x<T_{\ell}]}\exp(-A_{T_x})]}, & x>c,
		\ea\right.
		\]
		is the unique solution of 
	 \[
	 g(x)=a+ \kappa_r (s(x)-s(c))+\int_{\ell}^r (s(x \vee y)-s(c\vee y))g(y)\mu_A(dy),
	 \]
	 where $\kappa_r=\frac{d^+g_r(\ell+)}{ds}$. Moreover, $\kappa_r=0$ if $s(\ell)=-\infty$ or $A_{\zeta}=\infty$ a.s. on $[\zeta =T_{\ell}]$.
	 \item  The decreasing function
	 \[
	 g_{\ell}(x):=\left\{\ba{ll}
	 a E^x[\chf_{[T_c<T_{r}]}\exp(-A_{T_c})],& c< x,\\
	 \frac{a}{E^c[\chf_{[T_x<T_{r}]}\exp(-A_{T_x})]}, & x\leq c,
	 \ea
	 \right.
	 \]
	 is the unique solution of 
	 \[
	 g(x)=a+ \kappa_{\ell} (s(x)-s(c))+ \int_{\ell}^r (s(c \wedge y)-s(x\wedge y))g(y)\mu_A(dy).
	 \]
	 where $\kappa_{\ell}=\frac{d^-g_{\ell}(r-)}{ds}$. Moreover, $\kappa_{\ell}=0$ if $s(r)=-\infty$ or $A_{\zeta}=\infty$ a.s. on $[\zeta =T_{r}]$.
	\end{enumerate}
\end{theorem}
\begin{proof} Without loss of generality assume $a=1$.
	\begin{enumerate}
		\item  Let $y >x \vee c$. Suppose $c <x$. Then, 
		\[
		E^c[\chf_{[T_y<T_{\ell}]}\exp(-A_{T_y})]=E^c[\chf_{[T_x<T_{\ell}]}\exp(-A_{T_x})]E^x[\chf_{[T_y<T_{\ell}]}\exp(-A_{T_y})]
		\]
		by the strong Markov property. Via similar considerations when $x \leq c$, one thus arrives at
		\[
		g_r(x)=\frac{E^x[\chf_{[T_y<T_{\ell}]}\exp(-A_{T_y})]}{E^c[\chf_{[T_y<T_{\ell}]}\exp(-A_{T_y})]}.
		\]
		In particular, $g_r(X)\exp(-A)$ is a bounded martingale when stopped at $T_y$. Since $T_y$ increases to $\zeta$ as $y \rar r$, this shows $(g_r,A)$ is an \Ito-Watanabe pair bounded at $\ell$. Thus, it follows from Theorems \ref{t:main1}, \ref{t:main2} and \ref{t:uniqueIE} that $g_r$ is the unique solution of the stated equation and $\kappa_r=0$ when $s(\ell)=-\infty$. Moreover, $\frac{d^+g_r(\ell+)}{ds}=\frac{d^+g_r(\ell)}{ds}$ in view of (\ref{e:gsder}) and that $\mu_A$ does not charge $\{\ell\}$. 
		
		To prove the remaining claim  suppose $s(\ell)>-\infty$ but $A_{\zeta}=\infty$ a.s. on $[\zeta =T_{\ell}]$.  Then, (\ref{e:huil}) yields
		\be \label{e:expforder2}
		\begin{split}
			E^x[g(X_{T_b})]=&g(\ell)\frac{s(b)-s(x)}{s(b)-s(\ell)}+ \frac{s(x)-s(\ell)}{s(c)-s(\ell)}\\
			=& g(x) +\int_{\ell}^c  \frac{(s(x\wedge z)-s(\ell))(s(c)-s(x \vee z)}{s(c)-s(\ell)})g(z)\mu_A(dz).
		\end{split}
		\ee
		The above in particular implies $g(\ell)=0$ since 
		\be \label{e:Ainfatl}
		\int_{\ell}^x(s(z)-s(\ell))\mu_A(dz)=\infty
		\ee
		in view of Theorem \ref{t:Afinite} as $A_{\zeta}=\infty$ on $[\zeta=T_{\ell}]$.
		
		Moreover, (\ref{g:uilder}) leads to 
\[
\frac{d^+g_r(\ell)}{ds}=\frac{1}{s(c)-s(\ell)}-\int_{\ell}^c \frac{s(c)-s(z)}{s(c)-s(\ell)}g(z)\mu_A(dz).
\]	In particular,
\[
\infty>\int_{\ell}^c \frac{s(c)-s(z)}{s(c)-s(\ell)}g(z)\mu_A(dz)=\int_{\ell}^c \frac{(s(c)-s(z))(s(z)-s(\ell))}{s(c)-s(\ell)}\frac{g(z)}{s(z)-s(\ell)}\mu_A(dz).
\]
However, this implies  $\frac{d^+g(\ell)}{ds}=\lim_{z \rar \ell}\frac{g(z)}{s(z)-s(\ell)}=0$ in view of (\ref{e:Ainfatl}). 
		\item Repeat the above  starting with  $y<x\wedge c$ and observing
		\[
		g_{\ell}(x)=\frac{E^x[\chf_{[T_y<T_{r}]}\exp(-A_{T_y})]}{E^c[\chf_{[T_y<T_{r}]}\exp(-A_{T_y})]}.
		\]
	\end{enumerate}
\end{proof}
\begin{corollary} \label{c:Ainfinite}
	Suppose $P^x(A_{\zeta}=\infty)=1$ for all $x \in \elr$. Then, for any $a \in (0,\infty)$ and $c \in \elr$ the following hold:
	\begin{enumerate}
		\item The function
		\[
		g(x):=\left\{\ba{ll}
		aE^x[\exp(-A_{T_c})],& x\leq c,\\
		\frac{a}{E^c[\exp(-A_{T_x})]}, & x>c,
		\ea\right.
		\]
		is the unique solution of 
		\[
		g(x)=a+ \int_{\ell}^r (s(x \vee y)-s(c\vee y))g(y)\mu_A(dy).
		\]
		\item The function
		\[
		g(x):=\left\{\ba{ll}
		a E^x[\exp(-A_{T_c})],& c< x,\\
		\frac{a}{E^c[\exp(-A_{T_x})]}, & x\leq c,
		\ea
		\right.
		\]
		is the unique solution of 
		\[
		g(x)=a+ \int_{\ell}^r (s(c \wedge y)-s(x\wedge y))g(y)\mu_A(dy).
		\]
	\end{enumerate}
\end{corollary}
\begin{proof}
	This is a direct consequence of Theorem \ref{t:semiinf_s} since $E^x[\exp(-A_{T_c})]=E^x[\chf_{[T_c<T_{\ell}]}\exp(-A_{T_c})]=E^x[\chf_{[T_c<T_{r}]}\exp(-A_{T_c})]$.
\end{proof}
\begin{corollary}  \label{c:Anotinf}
	Suppose  $A$ is a  CAF, both $s(r)$ and $s(\ell)$ are finite, and $P^x(A_{\infty}<\infty)=1$. Let $g_r$ and $g_{\ell}$ be as in Therorem \ref{t:semiinf_s}. Then
		\[
		g_i(x):= a\frac{E^x[\chf_{[X_{\zeta}=i]}\exp(-A_{\zeta})]}{E^c[\chf_{[X_{\zeta}=i]}\exp(-A_{\zeta})]}, \quad i\in \{\ell,r\}.
		\]		
\end{corollary}
\begin{proof} Without loss of generality assume $a=1$.  As observed in the proof of Theorem \ref{t:semiinf_s}, 
	\[
	g_r(x)=\frac{E^x[\chf_{[T_y<T_{\ell}]}\exp(-A_{T_y})]}{E^c[\chf_{[T_y<T_{\ell}]}\exp(-A_{T_y})]}.
	\]
	for any $y >x\vee c$. Letting $y \rar r$ and observing that $X_{\zeta}=r$ on $[T_{r}<T_{\ell}]$ establish the claim. $g_{\ell}$ is handled similarly.
\end{proof}
So far in this paper the focus has been on semi-uniformly integrable subharmonic functions. The next result -- akin to the representation of solutions of ODEs in terms of linearly independent solutions -- shows that this is enough to characterise all.
\begin{theorem} \label{t:representation}
	For any $g\in \cS^+$  there exists a CAF $A$ with Revuz measure $\mu_A$ such that $g =\lambda_1 g_r +\lambda_2 g_{\ell}$, where $g_r$ and $g_{\ell}$ are as in Theorem \ref{t:semiinf_s}. 
\end{theorem}
\begin{proof}
	Since $g$ is subharmonic, it is a convex function of $s$ and there exists a CAF $B$ such that $g(X)-B$ is a $P^x$-local martingale for every $x\in \elr$. If $B\equiv 0$, $g$ must be an affine transformation of $s$, in which case $g$ is excessive and the claim holds with $\mu_A\equiv 0$. Note that if both $s(r)$ and $s(\ell)$ are infinite, that is $X$ is recurrent, only excessive functions are constants (see,.e.g., Exercise 10.39 in \cite{GTMP}). 
	
	Thus, suppose $B$ is not identically $0$. Since $g\in \cS^+$, $A_t =\int_0^t \frac{1}{g(X_s)}dB_s$ is well-defined as a CAF. As observed before, $g(X)\exp(-A)$ can be easily checked to be a $P^x$-local martingale. Let $g_r$ and $g_{\ell}$ be as defined in Theorem \ref{t:semiinf_s} with $a=1$.  Next consider an interval $(a^0,b^0)$ with $\ell <a^0<c<b^0<r$ and let $\lambda_1$ and $\lambda_2$ be such that
	\[
	\lambda_1 g_{r}(a^0)+ \lambda_2 g_{\ell}(a^0)= g(a^0) \mbox{ and } \lambda_1 g_r(b^0) + \lambda_2 g_{\ell}(b^0)=g(b^0)
	\]
	noting that the above has a unique solution since $g_{\ell}$ and $g_r$ are linearly independent. Since $\exp(-A) \left\{g(X)- \lambda_1 g_r(X)-\lambda_2 g_{\ell}(X)\right\}$ is a $P^x$-local martingale and $g$ as well as $g_i$s are continuous, one obtains
	\[
	g(x)-\lambda_1 g_r(x)-\lambda_2 g_{\ell}(x)=E^x\left[\exp(-A_{T_{ab}})\left\{g(X_{T_{ab}})- \lambda_1 g_r(X_{T_{ab}})-\lambda_2 g_{\ell}(X_{T_{ab}})\right\}\right]=0
	\]
	for any $x \in (a^0,b^0)$. Using the optional stopping theorem at $T_{zb^0}$ for $\ell<z <a^0$ and $x\in (a^0,b^0)$ shows    $g(z)=\lambda_1 g_r(z)+\lambda_2 g_{\ell}(z)$. Repeating the same argument at $T_{a^0z}$ for $z>b^0$ establishes $g(z)=\lambda_1 g_r(z)+\lambda_2 g_{\ell}(z)$ on $(b^0,r)$, hence the claim.
\end{proof}
One can turn the above result around to construct $g\in \cS^+$ starting with a Radon measure $\mu_A$ that can be associated to a CAF $A$ by solving first the equations for $g_r$ and $g_{\ell}$. However, the difficulty with this approach is that if $s(\ell)>-\infty$ and $A_{\zeta}<\infty$ on $[\zeta=T_{\ell}]$ (resp. $s(r)<\infty$ and $A_{\zeta}<\infty$ on $[\zeta=T_{r}]$), the right (resp. left) derivative of $g_r$ (resp. $g_{\ell}$) at $\ell$ (resp. $r$) is not known.  The next result offers a remedy to this problem.
\begin{theorem}
	Let $\mu_A$ be a Radon measure on $\elr$ and $A$ its corresponding CAF. Then the following statements are valid:
	\begin{enumerate}
		\item The increasing function $g_r$ of Theorem \ref{t:semiinf_s} is the unique solution of
		\[
		g(x)=a\frac{s(x)-s(\ell)}{s(c)-s(\ell)}-\int_{\ell}^c u(c;x,y)g(y)\mu_A(dy) +\int_{c+}^r (s(x\vee y)-s(y))g(y)\mu_A(dy),
		\]
		where
		\[
		u(c;x,y):=\lim_{a \rar \ell} \frac{(s(x\wedge y)-s(a))(s(c)-s(x \vee y))}{s(c)-s(a)}.
		\]
		\item The decreasing function $g_{\ell}$ of Theorem \ref{t:semiinf_s} is the unique solution of
		\[
		g(x)=a\frac{s(r)-s(x)}{s(r)-s(c)}-\int_{c+}^r u(x,y;c)g(y)\mu_A(dy) +\int_{\ell}^c (s(y)-s(x \wedge y))g(y)\mu_A(dy),
		\]
		where
			\[
		u(x,y;c):=\lim_{b \rar r}\frac{(s(x\wedge y)-s(c))(s(b)-s(x \vee y))}{s(b)-s(c)}.
		\]
	\end{enumerate}
\begin{proof}
	Only the first statement will be proven as the other can be shown by similar arguments. First note that if $s(\ell)=-\infty$ the stated equation coincides with the one in Theorem \ref{t:semiinf_s}.
	
	 Suppose $s(\ell)>-\infty$, which in turn yields $g_r(\ell)=0$. Indeed, for $x <c$, $g_r(x)\leq a P^x(T_c<T_{\ell})=a \frac{s(x)-s(\ell)}{s(c)-s(\ell)}$. Moreover, in view of (\ref{g:uilder}) and (\ref{e:gsder}), one has
	\[
	\kappa_r=\frac{a}{s(c)-s(\ell)}-\int_{\ell}^c \frac{s(c)-s(y)}{s(c)-s(\ell)}g_r(y)\mu_A(dy),
	\]
	which should equal $0$ in case  $\int_{\ell}^b (s(y)-s(\ell))\mu_A(dy)=\infty$, for some $b \in \elr$.
	
	A tedious but straightforward algebra  yields
	\[
	-(s(x)-s(c))\frac{s(c)-s(y)}{s(c)-s(\ell)} +s(x\vee y)-s(c\vee y)= -u(c;x,y)
	\]
	for $y\leq c$, hence the claim follows plugging above into the equation from Theorem \ref{t:semiinf_s}. 
	
\end{proof}	
\end{theorem}

\section{Path transformations via \Ito-Watanabe pairs} \label{s:transform}
This section is devoted to measure changes via \Ito-Watanabe pairs associated with semi-uniformly integrable subharmonic functions. Note that the pairs $(g,A)$ constructed in Corollary \ref{c:Anotinf} lead to bounded martingales, i.e. $g(X)\exp(-A)$ is bounded. Thus, the changes of measures via these martingales results in diffusion process whose laws are equivalent to that of the original diffusion. On the other hand, the local martingale associated to the \Ito-Watanabe pair of Corollary \ref{c:Ainfinite} is not necessarily a uniformly integrable martingale. 
\begin{proposition} \label{p:conv20}
	Suppose that $X$ is recurrent,  $f\geq 0$ and $A$ is a CAF such that $f(X)\exp(-A)$ is a supermartingale.  Assume further that $f$ is continuous on $\elr$ and either $f(\ell+)$ or $f(r-)$ exist (with the possibility of being infinite). Then, $f(X_t)\exp(-A_t)\rar 0$, $P^x$-a.s. for all $x \in \elr$.
\end{proposition}
\begin{proof} Since $f(X)\exp(-A)$ is a non-negative supermartingale, it converges a.s.. If this limit is non-zero with non-zero $P^x$-probability, then $P^x(\lim_{t \rar \infty}f(X_t)=\infty)>0$ since $A_{\infty}=\infty$, a.s.. However, this implies $\lim_{t\rar \infty}X_t$ exists and equals $\ell$ or $r$ with positive probability, which contradicts recurrence.
	\end{proof} 
Nevertheless,  one can still construct a Markov process, i.e. a {\em subprocess}, whose law is locally absolutely continuous with respect to that of the original process since $g(X)\exp(-A)$ is a {\em supermartingale multiplicative functional} (see Section 62 of \cite{GTMP}). 
\begin{theorem} \label{t:cofm}
	Consider an \Ito-Watanabe pair $(g,A)$, where $g$ is semi-uniformly integrable.   Then there exists a unique family of measures  $(Q^x)_{x \in \elr}$ on $(\Omega, \cF^u)$ rendering $X$ Markov with semigroup $(Q_t)_{t \geq 0}$ and $Q^x(X_0=x)=1$. Moreover, the following hold:
	\begin{enumerate}
		\item For every stopping time $T$ and $F\in \cF^*_T$
		\be \label{e:RN}
		Q^x(F,T<\zeta) =\frac{E^x[\chf_F \chf_{[T<\zeta]}g(X_T)\exp(-A_T)]}{g(x)} .
		\ee
		\item The semigroup $(Q_t)_{t \geq 0}$ coincides with that of a one-dimensional regular diffusion with no killing on $\elr$, scale function $s_g$ and speed measure $m_g$,
		where
		\[
		s_g(dx)=\frac{1}{g^2(x)}ds(x), \qquad m_g(dx)=g^2(x)m(dx).
		\]
		\item If $B$ is a CAF of $X$ with Revuz measure $\mu$ under $(P^x)_{x\in \elr}$ and speed measure $m$, its Revuz measure under $(Q^x)_{x\in \elr}$ and speed measure $m_g$ is given by $\mu_g(dx) =g^2(x)\mu(dx)$. 
	\end{enumerate} 
	
\end{theorem}
\begin{proof}
	The first statement follows directly from Theorem 62.19 in \cite{GTMP}.
	
	To prove the second statement observe that the killing measure on $\elr$ under $Q^x$ is null since there is no killing under $P^x$ and $g(X)\exp(-A)$ is a $P^x$-martingale when stopped at $T_{ab}$ for any $\ell<a<b<r$. 
	
	Moreover, $m_g$ is a symmetry measure for $(Q_t)$. Indeed, if $f$ and $h$ are bounded and measurable functions vanishing at $\Delta$, then
	\bean
	\int_{\ell}^r Q^x[f(X_t)]h(x)g^2(x)m(dx)&=&\int_{\ell}^r E^x[f(X_t)g(X_t)\exp(-A_t)]h(x)g(x)m(dx)\\
	&=&\int_{\ell}^r E^x[h(X_t)g(X_t)\exp(-A_t)]f(x)g(x)m(dx)\\
	&=&\int_{\ell}^r Q^x[h(X_t)]f(x)g^2(x)m(dx),
	\eean
	where the second equality follows from the fact that $m$ is the symmetry measure for $(P_t)$ and $\exp(-A)$ is a multiplicative functional in view of Theorem 13.25 in \cite{ChungWalsh}. Thus, $m_g$ is a speed measure associated to $(Q_t)_{t \geq 0}$. 
	
	Next let us observe that $s_g(X)$ is a $Q^x$-local martingale, where $s_g(x) =\int_c^x s_g(dx)$ for an arbitrary $c \in \elr$. However, this is equivalent to $s_g(X)g(X)\exp(-A)$ is a $P^x$-local martingale. That is, $(s_g g, A)$ has to be an \Ito-Watanabe pair. By killing $X$ at $T_a$ if necessary, this will follow from Theorem \ref{t:main2} if $s_g g$ solves (\ref{e:IEh_general}) on $(a,r)$ for any $a>\ell$ once $\ell$ is replaced by $a$.
	
	Indeed, redefining $s_g$ so that $s_g(a)=0$ one has via 
	\[
	\frac{d^+ s_g g}{ds}=\frac{1}{g} + s_g\frac{d^+g}{ds}
	\]
	and integration by parts that 
		\bean
		\frac{d^+ s_g g}{ds}(x)&=&\frac{1}{g(x)}+ s_g(x) \left(\frac{d^+g(a)}{ds}+ \int_a^xg(y)\mu_A(dy)\right)\\
		&=&\frac{1}{g(x)}+ \int_a^xs_g(y)g(y)\mu_A(dy)+ \int_a^x\frac{d^+g(y)}{ds}g^{-2}(y)\mu_A(dy)\\
		&=&\frac{1}{g(a)}+ \int_a^x s_g(y)g(y)\mu_A(dy),
	\eean
	where the first equality follows from (\ref{e:gsder}). Therefore,
	\bean
	s_g(x)g(x)&=&\frac{1}{g(a)}(s(x)-s(a))+ \int_a^x \int_a^z s_g(y)g(y)\mu_A(dy)ds(z)\\
	&=&\frac{1}{g(a)}(s(x)-s(a))+ \int_a^x (s(x)-s(y))s_g(y)g(y)\mu_A(dy),
	\eean
	which establishes that $(s_g g, A)$ is an \Ito-Watanabe pair in view of Theorem \ref{t:main2}.
	
Therefore,  once the speed measure $m_g$ is fixed, the associated scale function, $s^*$, will satisfy $s^*(dx)=k s_g(dx)$ for some $k>0$. Thus, the proof will be complete once it is shown that $k=1$. To this end, note that the potential density  of $X$ killed at $T_{ab}$ under the dynamics defined by $(Q_t)$ is given by $ku^*_{ab}$, where
\[
u^*_{ab}(x,y)= \frac{(s_g(x \wedge y)-s_g(a))(s_g(b)-s_g(x\vee y))}{s_g(b)-s_g(a)}. 
\]
To determine $k$, the quantity $Q^x(s(X_{T_{ab}}))-s(x)$ will be computed in two ways. First,
\be \label{Qxs1}
Q^x(s(X_{T_{ab}}))-s(x)= s(a)\frac{s_g(b)-s_g(x)}{s_g(b)-s_g(a)}+s(b)\frac{s_g(x)-s_g(a)}{s_g(b)-s_g(a)}-s(x).
\ee
On the other hand,
\bean
g(x) Q^x(s(X_{T_{ab}}))&=& E^x \left[s(X_{T_{ab}})g(X_{T_{ab}})\exp(-A_{T_{ab}})\right] \\
&=& g(x)s(x) + E^x\left[\int_0^{T_{ab}}\exp(-A_t)g'(X_t)dB_t\right]\\
&=&g(x)s(x) + g(x)Q^x\left[\int_0^{T_{ab}}\frac{g'(X_t)}{g(X_t)}dB_t\right],
\eean
where $B$ is as in Theorem \ref{t:ibp} and $g'$ stands for the left derivative of $g$ with respect to $s$. Since the Revuz measure of $B$ under $Q^x$ becomes $2g^2(x)s(dx)$ as will be shown below, one obtains
\be \label{Qxs2}
Q^x(s(X_{T_{ab}}))-s(x)=2k\int_a^b u_{ab}^*(x,y) g'(y) g(y)s(dy).
\ee
Now, combining (\ref{Qxs1}) and (\ref{Qxs2}) and repeating the similar calculations used in the proof of Theorem VII.3.12 in \cite{RY} yield
\[
\frac{ds}{ds_g}(x)-\frac{ds}{ds_g}(y)=2k\int_x^y g'(y)g(y)s(dy).
\]
However, the left hand side of the above is $g^2(x)-g^2(y)$ while the right hand side equals $k(g^2(x)-g^2(y))$. Thus, $k$ must equal $1$. 
	
Thus, it remains to prove the last statement. First, suppose $B_t:=\int_0^t f(X_s)ds$ for a non-negative measurable $f$. Then,
\[
Q^x(B_{T_{ab}})=\int_a^b u^*_{ab}(x,y)f(y) m_g(dy)=\int_a^b u^*_{ab}(x,y)f(y) g^2(y)m(dy),
\]
for any $\ell <a<b<r$, which implies the Revuz measure under  $(Q^x)$ given the speed measure $m_g$ equals $f(y)g^2(y)m(dy)$. Since the corresponding measure under $(P^x)$ is given by $f(y)m(dy)$, the claim follows for all such $B$. 

Moreover, by the occupation times formula $B_{T_{ab}}=\int_{\ell}^r L^y_{T_{ab}} f(y) m(dy)$, where $L^y$ is the local time of $X$ at level $y$ under $(P^x)$ with respect to $m$. Thus, for any non-negative measurable $f$
\[
\int_a^b u^*_{ab}(x,y)f(y) g^2(y)m(dy)=\int_{\ell}^r Q^x(L^y_{T_{ab}}) f(y) m(dy).
\]
On the other hand, $L^y$ is a CAF for $X$ under $(Q^x)$ and its support is contained in $\{y\}$ since $Q^x <<P^x$ on $\cF_t^*$  for every $t$ when restricted to $[t<\zeta]$. Then, by Proposition 68.1 in \cite{GTMP} $L^y$ is proportional to the local time at $y$ with respect to $m_g$ under $Q^x$. Therefore, $Q^x(L^y_{T_{ab}})= \alpha u^*_{ab}(x,y)$ for some $\alpha>0$, which can be easily seen equal to $g^2(y)$ in view of the above. This in turn implies the Revuz measure for $L^y$ is given by $g^2(y)\epsilon_y(dx)$, where $\epsilon_y$ is the Dirac measure at $y$. The proof is now complete since if $B$ is a CAF with  Revuz measure $\mu$ under  $(P^x)$ for the speed measure $m$, $B=\int_{\ell}^r\mu(dy)  L^y $.
\end{proof}
\begin{remark} \label{r:cofm}
	A quick inspection of the proof reveals that Theorem \ref{t:cofm} remain valid if $g =c_1 g_1 + c_2 g_2$, where $c_i\geq 0$ and $(g_i,A)$ are \Ito-Watanabe pairs with semi-uniformly integrable $g_i$s. Thus, it is valid for all \Ito-Watanabe pairs in view of Theorem \ref{t:representation}.
\end{remark}
Remarkably  \Ito-Watanabe pairs transform recurrent diffusions to transient ones. 
\begin{corollary} \label{c:trtransform1}
	Suppose $X$ is recurrent and $A$ is a CAF with $\mu_A(\bfE)>0$. Then $P^x(A_{\infty})=\infty)=1$. Consider $g=c_1 g_{r} + c_2 g_{\ell}$, where $g_{\ell}$ and $g_r$ are respectively the decreasing and increasing functions defined in Theorem \ref{t:semiinf_s}, $c_i\geq 0$ and $c_1+c_2>0$. Let $(Q^x)_{x\in \elr}$ denote the family of measures defined in Theorem \ref{t:cofm}. Then $X$ is transient under $(Q^x)_{x\in \elr}$. Moreover, 
		\begin{enumerate}
			\item If $c_1=0$, $Q^x(X_{\zeta -}=\ell)=1$.
			\item If $c_2=0$, $Q^x(X_{\zeta -}=r)=1$.
			\item If $c_1$ and $c_2$ are non-zero, $Q^x(X_{\zeta -}=r)>0$ and $Q^x(X_{\zeta -}=\ell)>0$.
		\end{enumerate}
\end{corollary}
\begin{proof} That $P^x(A_{\infty})=\infty)=1$ follows from Theorem \ref{t:Afinite}.
	
	Proof of the remaining statements will be given when $c_i$s do not vanish as the other cases are treated similarly. First observe that $g_r(r-)=\infty$. Indeed, if $g_r(r-)<\infty$, $g_r(X)\exp(-A)$ will be a bounded martingale with limit $0$ at infinity since $A_{\infty}$ is infinite. However, this would render $g_r(x)=0$ for all $x$ by the martingale property of $g_r(X)\exp(-A)$. Similarly, $g_{\ell}(\ell+)=\infty$.
	
	Suppose $ds(x)=dx$, without loss of generality, and note that
	\[
	s_g(\ell+)=-\int_{\ell}^c \frac{1}{g^2(x)}dx.
	\]
	Since $g$ is convex and $g(\ell+)=+\infty$, there exists $x^*<0$ and $k>0$ such that  $g(x)> -kx$ for all $x <x^*$. Thus, $s_g(\ell+)>-\infty$. Similarly, $s_g(r-)<\infty$. This proves the claim.
\end{proof}
The result above is a general case of the transient transformation considered in Proposition 5.1 in \cite{rectr}. A version  exists for transient diffusions as well.
\begin{corollary} \label{c:trtransform2}
	Suppose that $X$ is transient and $A$ is a CAF with $\mu_A(\bfE)>0$. Consider $g=c_1 g_{r} + c_2 g_{\ell}$, where $g_{\ell}$ and $g_r$ are respectively the decreasing and increasing functions defined in Theorem \ref{t:semiinf_s}, $c_i\geq 0$ and $c_1+c_2>0$. Let $(Q^x)_{x\in \elr}$ denote the family of measures defined in Theorem \ref{t:cofm}. Then $X$ is transient under $(Q^x)_{x\in \elr}$. Moreover, $Q^x(X_{\zeta-}=\ell)>0$ if $c_2>0$ and $Q^x(X_{\zeta-}=r)>0$ if $c_1>0$.
\end{corollary}
\begin{proof}
	Suppose $ds(x)=dx$ without loss of generality. If $\ell >-\infty$, $0< g_{\ell}(\ell+)<\infty$ if $A_{\zeta}<\infty$ on $[X_{\zeta-}=\ell]$, and $g_{\ell}(\ell+)=\infty$ on $A_{\zeta}=\infty$ on $[X_{\zeta-}=\ell]$. In the former case the finiteness of $s_g(\ell+)$ when $c_2>0$ is clear. Moreover, if $g_{\ell}(\ell+)=\infty$, $g_{\ell}(x)\leq (x-\ell)^{\frac{1}{4}}$ on $(\ell,x^*)$ for some $x^*$. This in turn implies $s_g(\ell+)<\infty$ if $c_2>0$. If, on the other hand, $\ell=-\infty$, $g_{\ell}(\ell+)=\infty$ and $g(x)> -kx$ for all $x <x^*$  for some $x^*<0$, which implies  the finiteness of $s_g(\ell+)$ when $c_2>0$.
	
	The implication of $c_1>0$ is proved similarly.
\end{proof}
\section{Examples}\label{s:examples}
\begin{example}[Connection with the fundamental solutions of ODEs]
	Suppose that $X$ is a solution of $dX_t=\sigma(X_t)dB_t + b(X_t)dt$, where $B$ is a standard Brownian motion and the coefficients $\sigma$ and $b$ are continuous. If $A_t=\int_0^t f(X_s)ds$, for some continuous and non-negative $f$, then $\mu_A(dx)=f(x)m(dx)=\frac{2f(x)}{\sigma^2(x)s'(x)}dx$.  Moreover, the increasing and decreasing functions of Theorem \ref{t:semiinf_s} can be easily shown to satisfy the ODE
	\[
	\half \sigma^2 g'' + b g'= fg.
	\]
\end{example}
\begin{example}[Soft borders in diffusion neighbourhoods]
	Consider a one-dimensional diffusion on natural scale with the state space $\bbR$. Let $\delta >0$ and note that $\mu_A(dx)=\frac{\epsilon_1(dx)}{\delta}$ is the Revuz measure for the CAF $(2\delta)^{-1}L^1$, where $L^1$ is the semimartingale local time for $X$ at $1$. One can solve (\ref{e:IEh}) explicitly in this case to find
	\[
	g_{r}(x)= \delta +(x-1)^+ \mbox{ and } g_{\ell}(x)=\delta +(1-x)^+
	\]
	satisfying $g_r(1)=g_{\ell}(1)=\delta$. Moreover, $(g_r,A)$ and $(g_{\ell},A)$ are \Ito-Watanabe pairs due to Theorem \ref{t:main1}.
	
	If one uses $(g_r,A)$ to apply a path transformation to $X$ via Theorem \ref{t:cofm}, one obtains a transient diffusion (see Corollary \ref{c:trtransform1}) with scale function
	\[
	s_{\delta}(x)=\left\{\ba{ll}
	\frac{1}{\delta}-\frac{1}{\delta+x-1}, & \mbox{if } x\geq 1;\\
	\frac{x-1}{\delta^2}, & \mbox{if } x<1.
	\ea	\right.
	\]
	Note that $s_{\delta}(\infty)<\infty$, implying that the diffusion drifts towards infinity in the long run. Moreover, the potential density $u_{\delta}$ is given by
	\[
	u_{\delta}(x,y)=\frac{1}{\delta}-s_{\delta}(x\vee y).
	\]
	In particular if the original $X$ is a Brown motion, the dynamics of $X$ under $(Q^x)$, where $Q^x$ is as defined in Theorem \ref{t:cofm}, is given by
	\[
	dX_t=dW_t+ \chf_{[X_t >1]}\frac{1}{\delta +X_t -1}dt,
	\]
	where $W$ is a standard Brownian motion. $X$ following the above dynamics still has the whole $\bbR$ as it state space. However, it can be guessed that for smaller values of $\delta$ it must be getting harder for $X$ to move from the half space $(1,\infty)$ to $(-\infty,1)$.  By taking formal limits as $\delta \rar 0$ one can see that $X$ is no longer regular: it is a Brownian motion on $(-\infty,1]$ while $X-1$ becomes a $3$-dimensional Bessel process on $[1,\infty)$. The set $\{1\}$ can be viewed as {\em soft border} between two regimes allowing transitions from the Brownian regime to the Bessel one but not the other way around. 
	
	This formal description can be made more rigorous by analysing $L^1_{\infty}$ - the cumulative local time spent at $1$ at the lifetime. It is well-known that $L^1_{\infty}$ is exponentially distributed under $Q^1$. For a fixed $\delta>0$ the parameter of this exponential distribution equals $\frac{s'_{\delta}(1)}{2 u_{\delta}(1,1)}=\frac{1}{2\delta}$. In particular $Q^1(L^1_{\infty})=2\delta\rar 0$ as $\delta$ tends to $0$. Moreover, for $x <1<y$, $Q^y(T_x<\infty)=1$ whereas
	\[
	Q^x(T_y<\infty)=\frac{u_{\delta}(x,y)}{u_{\delta}(y,y)}=\delta^2\frac{\delta +1-y}{\delta+x-1}\rar 0 \mbox{ as } \delta \rar 0.
	\]
	Thus, the transitions from $(-\infty,1)$ to $[1,\infty)$ continue as $\delta$ gets small. However, once the border $\{1\}$ is reached, $X$ is strongly pulled into the interior of $[1,\infty)$ and finds it increasingly difficult to get back to the border.
\end{example}

\begin{example}[Hard borders in diffusion neighbourhoods] \label{ex:hardborder}
As a continuation of the above example suppose $X$ is a standard Brownian motion but set $g=\lambda g_r +g_{\ell}
$ for some $\lambda >0$. In this case the path transformation via $(g,A)$ results in a transient diffusion that can go both $+\infty$ and $-\infty$ in the limit. Straightforward computations yield
	\[
s_{\delta}(x)=\left\{\ba{ll}
\frac{1}{\lambda}\left(\frac{1}{(1+\lambda)\delta}-\frac{1}{(1+\lambda)\delta+\lambda(x-1)}\right), & \mbox{if } x\geq 1;\\
\frac{1}{(1+\lambda)\delta+1-x}-\frac{1}{(1+\lambda)\delta}, & \mbox{if } x<1.
\ea	\right.
\]
Moreover, for $x<1<y$
\bean
Q^x(T_y<\infty)&=&\frac{ \delta(\delta(1+\lambda) +\lambda(y-1))}{( \delta(1+\lambda)+1-x)(\delta+ y-1)}\\ Q^y(T_x<\infty)&=&\frac{\delta(\delta(1+\lambda)+1-x)}{( \delta(1+\lambda)+\lambda (y-1))(\delta+ 1-x)}.
\eean
 Observe that both probabilities approach to $0$ as $\delta\rar 0$ indicating a {\em hard border} in the long run. 

On the other hand, for $x <1<y$, one has $Q^1(T_y<\infty)\rar \frac{\lambda}{1+\lambda}$ and $Q^1(T_x<\infty)=\frac{1}{1+\lambda}$. That is, if the process start at the hard border, it will end up in  the upper regime $(1,\infty)$ with probability $\lambda /1+\lambda$ as $\delta$ tends to $0$.
\end{example}

\begin{example}[Three communities with soft borders] In the last two examples are two distinct regimes in the limit. However, it is possible to divide the state space into more regions by using a mixture of local times at different levels.  
	
For instance in the setting of Example \ref{ex:hardborder} consider the CAF $A$ with the Revuz measure $\mu_A=\frac{1}{\delta}(\epsilon_{-1}+\epsilon_1)$. Then assuming $g_r(-1)=\delta$ and solving
\[
g_r(x)=\delta + (x \vee (-1)+1) + (x \vee 1-1)\frac{g_r(1)}{\delta}
\]
lead to $g_r(1)=\delta +2$ and, therefore,
\[
g_r(x)=\left\{\ba{ll}
\delta, & \mbox{if } x\leq -1;\\
x+1+\delta,  & \mbox{if } -1<x\leq 1;\\
\delta -\frac{2}{\delta} + x\frac{2(\delta+1)}{\delta}, & \mbox{if } x\geq 1,
\ea	\right.
\]
is an increasing subharmonic function such that $(g_r,A)$ is an \Ito-Watanabe pair. 

Similar considerations yield
\[
g_{\ell}(x)=\left\{\ba{ll}
\delta -(x+1)\frac{2(\delta+1)}{\delta} , & \mbox{if } x\leq -1;\\
\delta-(x+1)\frac{\delta}{\delta+2},  & \mbox{if } -1<x\leq 1;\\
\delta -\frac{2\delta}{\delta+2}& \mbox{if } x\geq 1,
\ea	\right.
\]
which is a decreasing subharmonic function such that $(g_{\ell},A)$ s an \Ito-Watanabe pair. 

Thus,  $g:=\frac{\delta}{\delta+2}g_r +g_{\ell}$ is constant on $(-1,1)$ and one should expect no change in the drift when $X$ belongs $(-1,1)$ after the application of Theorem \ref{t:cofm} via $(g,A)$. In fact an easy application of Girsanov's theorem shows that under $Q^x$
\[
dX_t=dW_t + b_{\delta}(X_t)dt,
\]
where $W$ is a $Q^x$-Brownian motion and 
\[
b_{\delta}(x)=\left\{\ba{ll}
\frac{-(\delta+1)}{\delta^2-1-x(\delta+1)}, & \mbox{if } x\leq -1;\\
0,  & \mbox{if } -1<x\leq 1;\\
\frac{\delta+1}{\delta^2-1+x(\delta+1)}& \mbox{if } x\geq 1.
\ea	\right.
\]
Clearly, as $\delta \rar 0$, the drift term converges to 
\[
b(x)=\left\{\ba{ll}
\frac{1}{x+1}, & \mbox{if } x\leq -1;\\
0,  & \mbox{if } -1<x\leq 1;\\
\frac{1}{x-1}& \mbox{if } x\geq 1.
\ea	\right.
\]
This indicates three regions $(-\infty, -1]$, $(-1,1)$, and $[1,\infty)$, where $X$ behaves like a Brownian motion in $(-1,1)$ and is able to move to the upper and lower regions. However, once $X$ enters  $(-\infty, -1]$ or $[1,\infty)$, it is not possible to exit these domains. More precisely, $-1-X$ and $X-1$ behave like $3$-dimensional Bessel processes on $(-\infty, -1]$ and $[1,\infty)$, respectively.  Following the terminology of the previous examples this can be regarded as three neighbourhoods with two soft borders. 
\end{example}
\section{Optimal stopping with random discounting} \label{s:OS}
Consider  the optimal stopping problem
\be \label{e:OSP}
V(x):=\sup_{\tau}E^x[e^{-A_{\tau}}f(X_{\tau})\chf_{[\tau<\zeta]}],
\ee
where $A$ is a CAF with Revuz measure $\mu_A$ with  $\mu_A(\bfE)>0$ and $f$ is continuous on $\bfE$.

The essence of the method that will be employed to solve the above problem is the removal of the discounting via a measure change that is developed in Section \ref{s:transform}. To this end set $g:=\lambda_1 g_r +\lambda_2 g_{\ell}$, where $g_r$ and $g_{\ell}$ are the increasing and decreasing functions defined in Theorem \ref{t:semiinf_s} and $\lambda_i$s are strictly positive constants. Then in view of Theorem \ref{t:cofm} and Remark \ref{r:cofm} there exists a unique family of measures $(Q^x)$ on $(\Omega,\cF^u)$ that renders $X$ a regular diffusion and satisfies (\ref{e:RN}). 

The next results gives a necessary condition for the finiteness of $V$.
\begin{proposition}
	Let $V$ be as defined in (\ref{e:OSP}). Then $V$ is finite only if $\frac{f}{g}$ is bounded on $\elr$. 
\end{proposition}
\begin{proof}
By (\ref{e:RN}) for any $\ell<a<x<b<r$
\[
E^x[e^{-A_{T_{ab}}}f(X_{T_{ab}})]=g(x)Q^x\left[\frac{f(X_{T_{ab}})}{g(X_{T_{ab})}}\right].
\]
If $\frac{f}{g}$ is unbounded, there is a sequence of $a_n \rar \ell$ and $b_n \rar r$ such that either $\frac{f(a_n)}{g(a_n)}$ or $\frac{f(b_n)}{g(b_n)}$ diverges to infinity. Moreover, Corollaries \ref{c:trtransform1} and \ref{c:trtransform2} imply that $Q^x(X_{\zeta-}=\ell)Q^x(X_{\zeta-}=r)>0$. Thus,
\[
\lim_{n \rar \infty}Q^x\left[\frac{f(X_{T_{a_n b_n}})}{g(X_{T_{a_n b_n}})}\right]=\infty.
\]
\end{proof}
Given the above necessary condition of finiteness solution of (\ref{e:OSP}) becomes equivalent to that of an optional stopping problem without discounting as shown in the following result, whose proof follows very closely the arguments used in the proof of Theorem 7.1 in \cite{rectr}.
\begin{theorem}
	\label{t:OS} Suppose $\frac{f}{g}$ is bounded on $\elr$. Then $V(x)=g(x) G(s_g(x))$, where $s_g$ is as in Theorem \ref{t:cofm} and $G$ is the smallest concave majorant of $\frac{f(s_g^{-1})}{g(s_g^{-1})}$. Moreover, 
		\[
	\tau_{\eps}^*:=\inf\left\{t\geq 0:\frac{f(X_t)}{g(X_t)}+\eps \geq G(s_g(X_t)) \right\}
	\]
	is $\delta$-optimal in the sense that for any $\delta >0$ there exists $\eps^*>0$ such that $E^x[e^{-A_{\tau^*_{\eps}}}f(X_{\tau^*_{\eps}})\chf_{[\tau_{\eps}^*<\zeta]}]> V(x)-\delta$ for all $\eps<\eps^*$.
	Furthermore, the stopping time  
		\[
	\tau^*:=\inf\left\{t\geq 0: f(X_t)\geq g(X_t)G(s_g(X_t)) \right\}
	\]
	is optimal if and only if $G(s_g(X_{\tau^*-}))\chf_{[\tau^*=\zeta]}=0$, $Q^x$-a.s..
\end{theorem}
\begin{proof}
It follows from Theorem 1 in Section 3.3 in \cite{Shiryaev} that 
\[
G(s_g(x))=\sup_{\tau}Q^x\left[\chf_{[\tau<\zeta]}\frac{f(X_{\tau})}{g(X_{\tau})}\right]
\]
since excessive functions are concave functions of $s_g$. Note that since $G$ is concave and $s_g$ is increasing, $G(s_g)$ can be defined by continuity at $\ell$ and $r$.  

Moreover,  Lemma 8  in Section 3.2 and Theorem 2 in  Section 3.3 of \cite{Shiryaev} yield
\be \label{e:Gepsopt}
G(s_g(x))=Q^x\left[G(s_g(X_{\tau^*_{\eps}}))\right]
\ee
and  $Q^x(\tau_\eps<\zeta)=1$ due to the fact that $x \mapsto \frac{f(x)}{g(x)}\chf_{x\in \elr}$ is lower semicontinuous. Thus,
\be \label{e:epslbd}
v(x)=E^x[e^{-A_{\tau_{\eps}^*}}v(X_{\tau_{\eps}^*})\chf_{[\tau_{\eps}^*<\zeta]}],
\ee
where $v:=g G(s_g)$.

The first consequence of the above is that $V=v$. Indeed,  
\[
E^x[e^{-A_{\tau}}f(X_{\tau})\chf_{[\tau<\zeta]}]=g(x)Q^x\left[\frac{f(X_{\tau})}{g(X_{\tau})}\chf_{[\tau<\zeta]}\right]\leq v(x).
\]
That is, $v$ is an upper bound for $V$. On the other hand, 
\be\label{e:epsoptV}
V(x)\geq  E^x[e^{-A_{\tau^*_{\eps}}}f(X_{\tau^*_{\eps}})\chf_{[\tau_{\eps}^*<\zeta]}]=g(x)Q^x\left[\frac{f(X_{\tau^*_{\eps}})}{g(X_{\tau^*_{\eps}})}\right]\geq g(x)Q^x[G(s_g(X_{\tau^*_{\eps}}))-\eps].
\ee
Consequently, $V(x)\geq v(x)$ by letting $\eps \rar 0$ in view of (\ref{e:Gepsopt}). This establishes that $v=V$ and $\tau_{\eps}^*$ is $\delta$-optimal at once. 

Moreover, if  $\tau^*$ is optimal,
\[
V(x)= E^x[e^{-A_{\tau^*}}f(X_{\tau^*})\chf_{[\tau^*<\zeta]}]=g(x) Q^x[G(s_g(X_{\tau^*}))\chf_{[\tau^*<\zeta]}].
\]
However,
\[
Q^x[G(s_g(X_{\tau^*}))\chf_{[\tau^*<\zeta]}]=Q^x[G(s_g(X_{\tau^*-}))]-Q^x[G(s_g(X_{\tau^*-}))\chf_{[\tau^*=\zeta]}]=G(s_g(x))-Q^x[G(s_g(X_{\tau^*-}))\chf_{[\tau^*=\zeta]}],
\]
where the last equality is following from the fact that $G(s_g(X))$ is a bounded $Q^x$-local martingale on $[0,\tau^*)$. Thus, $G(s_g(X_{\tau^*-}))\chf_{[\tau^*=\zeta]}=0$, $Q^x$-a.s..

Conversely, if $G(s_g(X_{\tau^*-}))\chf_{[\tau^*=\zeta]}=0$, $Q^x$-a.s., $G(s_g(x))=Q^x[G(s_g(X_{\tau^*-}))]=Q^x[G(s_g(X_{\tau^*}))\chf_{[\tau^*<\zeta]}]$, which implies $\tau^*$ is optimal. 
 
\end{proof}
\bibliographystyle{siam}
\bibliography{../ref.bib}
\appendix
\section{Appendix}
 `\begin{proof}[Proof of Lemma \ref{l:mainR}] Proof will consider $v_c(x,y)=s(x \vee y)-s(c\vee y)$ and the other case is handled similarly.
 	\begin{enumerate}
 		\item  Note that one can choose $c$ such that $a<c<x$ without loss of generality. Then,
 		\bean
 		E^x[g(X_{T_{ab}})]&=& g(c) + \kappa (s(x)-s(c)) \int_{\ell}^r (E^x [s(X_{T_{ab}}\vee y)]-s(c\vee y))g(y)\mu_A(dy) \\
 		&=& g(c)+  \kappa (s(x)-s(c))+ \int_{\ell}^a (E^x s(X_{T_{ab}})-s(c))g(y)\mu_A(dy)\\
 		&&+\int_{a+}^b (E^x s(X_{T_{ab}}\vee y)-s(c\vee y))g(y)\mu_A(dy)\\
 		&=& g(c)+  \kappa (s(x)-s(c))+\int_{\ell}^{a} (s(x)-s(c))g(y)\mu_A(dy) \\
 		&&+\int_{a+}^b (E^x [s(X_{T_{ab}}\vee y)]-s(c\vee y))g(y)\mu_A(dy),
 		\eean
 		where the first equality is due to the fact that $s(z\vee y)=s(z)$ for $y<a$ and $s(z\vee y)=s(y)$ whenever $z\in (a,b)$. On the other hand, for $y \in [a,b]$,
 		\bean
 		E^x[s(X_{T_{ab}}\vee y)]&=&s(y) P^x(T_a<T_b)+ s(b) P^x(T_b<T_a)= s(y)\frac{s(b)-s(x)}{s(b)-s(a)}+ s(b) \frac{s(x)-s(a)}{s(b)-s(a)}\\
 		&=&s(x)+ \frac{(s(y)-s(a))(s(b)-s(x))}{s(b)-s(a)}.
 		\eean
 		Therefore,
 		\bean
 		E^x[g(X_{T_{ab}})]&=& g(c) + \kappa (s(x)-s(c))+\int_{\ell}^x (s(x\vee y)-s(c\vee y))g(y)\mu_A(dy)\nn \\
 		&& +\int_{x+}^b \left(s(x)-s(y)-\frac{(s(y)-s(a))(s(b)-s(x))}{s(b)-s(a)}\right)g(y)\mu_A(dy)\nn\\
 		&&\int_{a+}^x\frac{(s(y)-s(a))(s(b)-s(x))}{s(b)-s(a)}g(y)\mu_A(dy)\nn\\
 		&=&g(x)+ \int_{x+}^b \frac{(s(x)-s(a))(s(b)-s(y))}{s(b)-s(a)}g(y)\mu_A(dy)\nn\\
 		&&+\int_{a+}^x\frac{(s(y)-s(a))(s(b)-s(x))}{s(b)-s(a)}g(y)\mu_A(dy)\nn\\
 		&=&g(x)+ \int_a^b \frac{(s(x\wedge y)-s(a))(s(b)-s(x\vee y))}{s(b)-s(a)}g(y)\mu_A(dy),
 		\eean
 		where the second equality is due to (\ref{e:IEh}), the third follows from that $y \mapsto \frac{(s(y)-s(a))(s(b)-s(x))}{s(b)-s(a)}$ vanishes at $y=a$.
 		\item It is clear from the definition that $g$ is continuous on $\elr$. Thus, both $O^+$ and $O^-$ are open. Let $ x \in O^+$ and consider a neighborhood around $x$ with left endpoint $a$ and right endpoint $b$ such that $(a,b)\subset O^+$. Then, (\ref{e:EhabR}) yields
 		\[
 		E^x[g(X_{T_{ab}})]\geq g(x),
 		\]
 		as a consequence of the strict positivity of $g$ on $O^+$. This proves that $g$ is $s$-convex on $O^+$ since
 		\[
 		E^x[g(X_{T_{ab}})]= g(a)\frac{s(b)-s(x)}{s(b)-s(a)}+ g(b)\frac{s(x)-s(a)}{s(b)-s(a)}.
 		\]
 		The same technique can be used to prove $g$ is $s$-concave on $O^-$. 
 		\item Suppose $s(\ell)=-\infty$ and $v_c(x,y)= s(x\vee y)-s(c\vee y)$. First observe that the integrability assumption $\int_l^r |s(x\vee y)-s(c \vee y)||g(y)|\mu_A(dy)<\infty$ implies
 		\[
 		\int_l^{x} |g(y)|\mu_A(dy)<\infty
 		\]
 		for any $x \in \elr$. 
 		
 		Moreover, for any $a<x<b$,
 		\[
 		\frac{(s(x\wedge y)-s(a))(s(b)-s(x\vee y))}{s(b)-s(a)}\leq s(b)-s(x)
 		\]
 		for all $y \in (a,b)$. Thus, the dominated convergence theorem applied to (\ref{e:EhabR}) yields
 		\[
 		(s(b)-s(x))\lim_{a\rar \ell}\frac{g(a)}{s(b)-s(a)}+g(b)= \lim_{a\rar \ell}E^x[g(X_{T_{ab}})]=g(x)+\int_{\ell}^b (s(b)-s(x\vee y))g(y)\mu_A(dy).
 		\]
 		Note in particular that if $g$ is u.i. near $\ell$, $\lim_{a\rar \ell}E^x[g(X_{T_{ab}})]=g(b)$ and, consequently, $\lim_{a\rar \ell}\frac{g(a)}{s(b)-s(a)}=0$. 
 		
 		Thus,
 		\[
 		\lim_{a\rar \ell}\frac{g(a)}{s(b)-s(a)}=\lim_{x\rar \ell}\frac{g(x)}{s(b)-s(x)}+\lim_{x\rar \ell}\int_{\ell}^b \frac{s(b)-s(x\vee y)}{s(b)-s(x)}g(y)\mu_A(dy),
 		\]
 		which in turn yields 
 		\[
 		\lim_{x\rar \ell}\int_{\ell}^b \frac{s(b)-s(x\vee y)}{s(b)-s(x)}g(y)\mu_A(dy)=0.
 		\]
 		
 		On the other hand, (\ref{e:genelIE}) implies
 		\be \label{e:gblim}
 		0=\lim_{x\rar \ell}\frac{g(x)}{s(b)-s(x)}+ \kappa+ \lim_{x\rar \ell}\int_{\ell}^b \frac{s(b)-s(x\vee y)}{s(b)-s(x)}g(y)\mu_A(dy),
 		\ee
 		which establishes $\kappa=-\lim_{x\rar \ell}\frac{g(x)}{s(b)-s(x)}$. 
 		
 		The second assertion follows form the fact that $\lim_{a\rar \ell}\frac{g(a)}{s(b)-s(a)}=0$ when $g$ is u.i. near $\ell$ as observed above. 
 		
 		\item If $g$ changes its sign,  there exists a $c^* \in (\ell,r)$ such that either $g$ is decreasing, $s$-convex on $(\ell,c^*)$ and $s$-concave on $(c^*,r)$ or increasing, $s$-concave on $(\ell,c^*)$ and $s$-convex on $(c^*,r)$. Since $-g$ also solves (\ref{e:genelIE}), assume without loss of generality that the former case holds. Fix $c \in (\ell,c^*)$ and let $x \in (c,c^*)$ be arbitrary. Then, assuming $v_c(x,y)= s(x\vee y)-s(c\vee y)$ 
 		\[
 		g(x)= g(c) +  \int_{\ell}^x( s(x)-s(c\vee y))g(y)\mu_A(dy) \geq g(c)
 		\]
 		since $g$ is non-negative on $(\ell,c^*)$. This shows $g$ is increasing on $(l,c^*)$ yielding a contradiction. 
 		
 		Similarly, if $v_c(x,y)=s(c\wedge y)-s(x\wedge y)$, let $c^*<c<x$ and note that $g$ is nonpositive on $(c^*,r)$. Then,
 		\[
 		g(x)= g(c) +  \int_{x+}^r( s(c)-s(x \wedge y))g(y)\mu_A(dy) \geq g(c)
 		\]
 		contradicts that $g$ is decreasing.
 		
 		\item If $g \geq 0$ and $x>c$, then
 		\[
 		g(x)-\kappa s(x)=g(c)-\kappa s(c)+\int_{\ell}^x (s(x)-s(c\vee y))g(y)\mu_A(dy)\geq g(c)
 		\]
 		since $g\geq0$ and $s$ is increasing.
 		\item  Note that, for sufficiently small $h>0$ and  $x\in [\ell,r)$ such that $g(x)<\infty$,
 		\bean
 		g(x+h)-g(x)&=&\kappa (s(x+h)-s(x))+(s((x+h))-s(x))\int_{\ell}^x g(y)\mu_A(dy)\\
 		&&+\int_{x+}^{x+h}(s(x+h)-s(y))g(y)\mu_A(dy),
 		\eean
 		which in turn yields 
 		\[
 		\frac{d^+g(x)}{ds}=\kappa+ \int_{\ell}^x g(y)\mu_A(dy)
 		\]
 	since $g$ is continuous, $\mu_A$ is finite on any small neighbourhood around $x$ and does not charge $\{\ell\}$. 
 	
 	Similarly,
 	\[
 	g(x)-g(x-h)=\kappa(s(x)-s(x-h)) +(s(x)-s(x-h))\int_{\ell}^{x-h}g(y)\mu_A(dy) + \int_{(x-h)+}^{x}(s(x)-s(y))g(y)\mu_A(dy),
 	\]
 	and therefore 
 	\[
 	\frac{d^-g(x)}{ds}=\kappa + \int_{\ell}^{x-} g(y)\mu_A(dy).
 	\]
 	The other case for $v_c$ is handled in the same manner.
 	\end{enumerate}

 \end{proof}
\begin{proof}[Proof of Theorem \ref{t:main2}]
	\Implies{t:main:eq1}{t:main:eq2}:
	Suppose $s(\ell)>-\infty$. Then  $g(\ell)<\infty$ since $g$ is uniformly integrable near $\ell$.  Consider, as suggested above, $\tilde{g}(x)=g(x)-g(\ell)-\frac{d^+g(\ell)}{ds}(s(x)-s(\ell))$ and note that $\tilde{g}$ is  $s$-convex with $\tilde{g}(\ell)=\frac{d^+g(\ell+)}{ds}=0$. Moreover, $\tilde{g}(X)\exp(-B)$ is a $P^x$-local martingale for any $x \in (\ell,r)$, where
	\[
	dB_t=\frac{g(X_t)}{\tilde{g}(X_t)}dA_t.
	\]
	This in particular implies $\mu_B(dy)=\frac{g(y)}{\tilde{g}(y)}\mu_A(dy)$.  Moreover, thanks to Proposition \ref{p:hhp0},
	\[
	\tilde{g}(x)=\int_{\ell}^x (s(x)-s(y))\tilde{g}(y)\mu_B(dy)=\int_{\ell}^x (s(x)-s(y))g(y)\mu_A(dy).
	\]
	Thus,
	\[
	g(x)= g(\ell)+\frac{d^+g(\ell+)}{ds}(s(x)-s(\ell))+ \int_{\ell}^x (s(x)-s(y))g(y)\mu_A(dy),
	\]
	which implies (\ref{e:IEh_general}). 
	
	If $s(r)<\infty$ and $g$ is u.i. near $r$, then define $\tilde{g}(x):= g(x)-g(r) + \frac{d^-g(r-)}{ds}(s(r)-s(x))$ and proceed along the above lines to arrive at (\ref{e:IEh_general}).

	\Implies{t:main:eq2}{t:main:eq1}: Suppose $s(\ell)>-\infty$ and observe that $g(\ell)<\infty$ as a consequence. Moreover, in view of  (\ref{e:gsder})
	\[
	\frac{d^+g(\ell+)}{ds}= \kappa.
	\]
	Thus,  $\tilde{g}(x):=g(x)-g(\ell)-\frac{d^+g(\ell+)}{ds}(s(x)-s(\ell))$ is  non-negative and $s$-convex, and  there exists a CAF $B$ such that $\tilde{g}(X)-B$ is a $P^x$-local martingale for each $x \in \elr$. Proceeding as in the second part of the proof of Theorem \ref{t:main1} one obtains for $\ell <a<x<b <r$
	\[
	E^x[\tilde{g}(X_{T_{ab}})]=\tilde{g}(x)+ E^x[B_{T_{ab}}]=\tilde{g}(x)+ \int_a^b \frac{(s(x\wedge y)-s(a))(s(b)-s(x\vee y))}{s(b)-s(a)}\mu_B(dy).
	\]
	Next observe that 
	\[
	E^x[\tilde{g}(X_{T_{ab}})]-\tilde{g}(x)=E^x[g(X_{T_{ab}})]={g}(x)
	\]
	since $s(X)$ is a bounded martingale when stopped at $T_{ab}$. Combining this with (\ref{e:EhabR}) yields
	$\mu_B(dy)=g(y)\mu_A(dy)$ by the same argument in the proof of Theorem \ref{t:main1}. Consequently, $g(X)-\int_0^{\cdot}g(X_t)dA_t$ is a local martingale, which entails $g(X)\exp(-A)$ is a $P^x$-local martingale. Since $g$ is continuous and $g(\ell)<\infty$, $g$ is obviously u.i. near $\ell$.
	
	The case $s(r)<\infty$ is handled in the same manner.
\end{proof}
\end{document}